\documentclass[final]{siamart190516}
\usepackage{amsfonts}
\usepackage{amsfonts,amsmath,amssymb}
\usepackage{mathrsfs,mathtools,stmaryrd,wasysym}
\usepackage{enumerate}
\usepackage{enumitem}
\usepackage{esint}
\usepackage{graphicx,psfrag}
\usepackage{hyperref}
\usepackage{amsmath,stackengine}
\usepackage{yfonts}

\usepackage{pgfplots}
\usepgfplotslibrary{groupplots}
\pgfplotsset{compat=newest}
\usepgfplotslibrary{external}
\usepgfplotslibrary{colorbrewer}

\newsiamremark{remark}{Remark}
\newcommand{\TheTitle}{A DPG method for linear quadratic optimal control problems}
\newcommand{\ShortTitle}{DPG method for optimal control problems}
\newcommand{\TheAuthors}{T. F\"uhrer, F. Fuica}

\headers{\ShortTitle}{\TheAuthors}

\title{{\TheTitle}\thanks{TF is partially supported by ANID through FONDECYT Project 1210391. FF is supported by ANID through FONDECYT postdoctoral project 3230126.}}

\author{Thomas F\"uhrer\thanks{Facultad de Matem\'aticas, Pontificia Universidad Cat\'olica de Chile, Avenida Vicu\~{n}a Mackenna 4860, Santiago, Chile.
(\email{tofuhrer@mat.uc.cl}).}
\and
Francisco Fuica\thanks{Facultad de Matem\'aticas, Pontificia Universidad Cat\'olica de Chile, Avenida Vicu\~{n}a Mackenna 4860, Santiago, Chile.
(\email{francisco.fuica@mat.uc.cl}).}}

\ifpdf
\hypersetup{
  pdftitle={\TheTitle},
  pdfauthor={\TheAuthors}
}
\fi

\date{Draft version of \today.}

\begin{document}

\maketitle

\begin{abstract}
  The DPG method with optimal test functions for solving linear quadratic optimal control problems with control constraints is studied. We prove existence of a unique optimal solution of the nonlinear discrete problem and characterize it through first order optimality conditions. Furthermore, we systematically develop a priori as well as a posteriori error estimates. Our proposed method can be applied to a wide range of constrained optimal control problems subject to, e.g., scalar second-order PDEs and the Stokes equations. Numerical experiments that illustrate our theoretical findings are presented.
\end{abstract}

\begin{keywords}
optimal control, discontinuous Petrov--Galerkin method, optimal test functions, finite elements, convergence, error estimates.
\end{keywords}

\begin{AMS}
49J20,   	   
49M25,		   
65N15,         
65N30.         
\end{AMS}


\section{Introduction}\label{sec:intro}

The analysis of formulations and approximation techniques for optimal control problems has been a matter of interest in the last decades. While problems without control constraints are fairly well established, the development of new approximation techniques for control problems with constraints is far from complete \cite{MR2516528}. The main source of difficulty of these control-constrained problems is its nonlinear feature inherited primarily by the considered restrictions. Various solution techniques have been proposed and analyzed in the literature. For an overview and an up-to-date discussion on this matter, we refer the reader to \cite{MR2516528,MR2583281,MR3308473} and references therein. 

The discontinuous Petrov-Galerkin method with optimal test-functions (DPG) is a class of minimal residual methods that approximates solutions of, e.g., PDEs, by minimizing residuals in dual norms of broken (Sobolev) spaces \cite{MR2630162,MR2743600,MR2837484,MR3521055}. One of the main goals for the development of the DPG method was to provide numerical methods that robustly control the error, i.e., independent of the perturbation parameter in singularly perturbed problems,~\cite{MR3095479,MR3279496}.  Another notable feature in the DPG setting is the possibility to analyze the use of different variational formulation where the most prominent example is the ultra-weak formulation. Due to the use of broken test spaces the DPG method provides a natural localized error estimator~\cite{MR3215064} that can be used to steer adaptive algorithms.
A posteriori error estimators are important tools in the analysis of finite element methods for optimal control problems, see, e.g.,~\cite{MR3212590} and references therein.

In contrast to the aforementioned advances and to the best of our knowledge, this is the first work that develops a framework for solving optimal control problems based on DPG methods.
Constraints on the control variable are included leading to nonlinear optimization problems.
Under the paradigm of \emph{first discretize, then optimize} \cite[chapter 3]{MR2516528} together with standard assumptions for the analysis of DPG methods, existence of continuous and discrete solutions is shown. A priori and a posteriori error estimates for the general framework are derived. We propose an a posteriori error estimator which, depending on the problem under consideration, requires to localize residuals similar as in~\cite{MR4090394} and~\cite{MR3980263}.
Our proposed method involves the discretization of a dual variable (adjoint state) and in context of minimal residual methods this may lead to reduced convergence rates as has been studied in, e.g.,~\cite{MR4334615}.
However, we prove optimal convergence under minimal regularity assumptions provided that the data is sufficiently smooth.
Here, as particular examples we consider optimal control problems subject to the Poisson equation and Stokes equations.
While we do not claim that our proposed method competes with existing finite element methods for these problems, we emphasize that the intention of the present article is to lay foundations for future work involving more challenging problems. 
In particular, we are interested in optimal control problems subject to singularly perturbed problems, see, e.g.~\cite{MR2647025,MR2495058}, or parabolic equations. Least-squares finite element methods for the latter problem have been recently considered in~\cite{LSQoptimal,OptControlHeatGS22}.

The outline of this work is as follows. We set notation and recall some preliminaries for the DPG method in section \ref{sec:not_and_prel}. In section \ref{sec:the_OCP}, we present the general linear quadratic optimal control problem \eqref{def:opt_cont_prob}. Moreover, we show existence of a unique optimal solution and optimality conditions. Section \ref{sec:discrete_approx} presents a finite element discretization based on the DPG method for problem  \eqref{def:opt_cont_prob}. We also prove a priori and a posteriori error estimates for the proposed discretization scheme. In section \ref{sec:examples}, we show how to apply the developed method for some particular problems. We end in section \ref{sec:num_examples}, where we provide some numerical examples that illustrates our theory.


\section{Preliminaries}\label{sec:not_and_prel}

In this section we introduce notation for spaces and operators used throughout this work.

\subsection{Notation for spaces}\label{sec:notation}

We shall use standard notation for Lebesgue and Sobolev spaces and their norms. Let $\mathfrak{X}$ and $\mathfrak{Y}$ be Hilbert spaces. We denote by $(\cdot,\cdot)_{\mathfrak{X}}$ and $\|\cdot\|_{\mathfrak{X}}$ the inner product and norm in $\mathfrak{X}$, respectively. We denote by $\mathcal{I}_{\mathfrak{X}}^{\mathfrak{X}'}$ the canonical isomorphism of $\mathfrak{X}$ onto $\mathfrak{X}'$.
Given a linear bounded operator $\mathfrak{B}\colon \mathfrak{X}\to\mathfrak{Y}$ we denote by $\mathfrak{B}^*\colon \mathfrak{Y}'\to \mathfrak{X}'$ its dual.
In slight abuse of notation, throughout this work, any bidual of a Hilbert space $\mathfrak{X}$ will be identified with $\mathfrak{X}$.


\subsection{Broken Sobolev spaces}\label{sec:broken_Sob_spaces}

In this section we introduce \emph{broken Sobolev spaces} and \emph{trace spaces}. To present these spaces, we consider  a bounded polytopal Lipschitz domain $\Omega\subset\mathbb{R}^{d}$ ($d\geq 2$) and a \emph{shape-regular} mesh $\mathscr{T}$ of open Lipschitz elements $T$ on $\Omega$ with boundaries $\partial T$, so that any two elements do not intersect and $\overline{\Omega} = \cup_{T\in\mathscr{T}}\overline{T}$. We thus introduce the infinite-dimensional (but mesh-dependent) spaces \cite[Section 2]{MR3521055}
\begin{equation*}
H^1(\mathscr{T}):=\prod_{T\in\mathscr{T}}H^1(T), \qquad H(\text{div},\mathscr{T}):=\prod_{T\in\mathscr{T}}H(\text{div},T),
\end{equation*}
with norms $\|\cdot\|_{H^1(\mathscr{T})}^2:=\sum_{T\in\mathscr{T}}\|\cdot\|_{H^1(T)}^2$ and $\|\cdot\|_{H(\text{div},\mathscr{T})}^2:=\sum_{T\in\mathscr{T}}\|\cdot\|_{H(\text{div},T)}^2$, respectively. 

For any $T\in\mathscr{T}$ we define two \emph{local trace operators}, namely, $\text{tr}_{T}^{\text{grad}}: H^1(T)\to H(\text{div},T)'$ and $\text{tr}_{T}^{\text{div}}:H(\text{div},T)\to H^1(T)'$ by
\begin{align*}
\langle \text{tr}_{T}^{\text{grad}}(z), \boldsymbol{\delta\!\tau}\rangle_{\partial T}:&=(z,\text{div }\boldsymbol{\delta\!\tau})_{L^2(T)} + (\nabla z, \boldsymbol{\delta\!\tau})_{L^2(T)} \quad \left(z\in H^1(T), ~ \boldsymbol{\delta\!\tau}\in H(\text{div},T)\right), \\
\langle \text{tr}_{T}^{\text{div}}(\boldsymbol{\tau}), \delta\!z\rangle_{\partial T}:&=(\boldsymbol{\tau},\nabla\delta\!z)_{L^2(T)} + (\text{div } \boldsymbol{\tau}, \delta\!z)_{L^2(T)} \quad \left(\boldsymbol{\tau}\in H(\text{div},T),~ \delta\!z\in H^1(T)\right).
\end{align*}
The corresponding (global) product versions of the previous trace operators are $\text{tr}^{\text{grad}}: H^1(\mathscr{T})\to H(\text{div},\mathscr{T})'$ and $\text{tr}^{\text{div}}:H(\text{div},\mathscr{T})\to H^1(\mathscr{T})'$, which are defined by
\begin{align*}
\langle \text{tr}^{\text{grad}}(z), \boldsymbol{\delta\!\tau}\rangle_{\partial\mathscr{T}}:&=\sum_{T\in\mathscr{T}}\langle \text{tr}_{T}^{\text{grad}}(z), \boldsymbol{\delta\!\tau}\rangle_{\partial T} \qquad \left(z\in H^1(\mathscr{T}),~\boldsymbol{\delta\!\tau}\in H(\text{div},\mathscr{T})\right), \\
\langle \text{tr}^{\text{div}}(\boldsymbol{\tau}), \delta\!z\rangle_{\partial \mathscr{T}}:&=\sum_{T\in\mathscr{T}}\langle \text{tr}_{T}^{\text{div}}(\boldsymbol{\tau}), \delta\!z\rangle_{\partial T} \qquad \left(\boldsymbol{\tau}\in H(\text{div},T)),~ \delta\!z\in H^1(T)\right).
\end{align*}
These operators give rise to the trace spaces
\begin{equation*}
H_{0}^{1/2}(\partial\mathscr{T}):=\text{tr}^{\text{grad}}(H_{0}^{1}(\Omega)), \quad H^{-1/2}(\partial\mathscr{T}):=\text{tr}^{\text{div}}(H(\text{div},\Omega)),
\end{equation*}
which are endowed, respectively, with the norms
\begin{align*}
\|\hat{z}\|_{1/2,\partial\mathscr{T}}:&=\inf \{\|z\|_{H^1(\Omega)}: z\in H^1(\Omega),~ \text{tr}^{\text{grad}}(z)=\hat{z}\},\\
\|\hat{\tau}\|_{-1/2,\partial\mathscr{T}}:&=\inf \{\|\boldsymbol\tau\|_{H(\text{div},\Omega)}: \boldsymbol\tau\in H(\text{div},\Omega),~ \text{tr}^{\text{div}}(\boldsymbol\tau)=\hat{\tau}\}.
\end{align*}

For further properties of the previous spaces, we refer the reader to \cite[Section 2]{MR3521055}.


\subsection{Approximation spaces}\label{sec:approximation_spaces}

Let $\Omega\subset\mathbb{R}^{d}$ ($d\geq 2$) and $\mathscr{T}$ be as in section \ref{sec:broken_Sob_spaces}. Here, we assume that each element $T\in\mathscr{T}$ is a simplex. Given $T\in \mathscr{T}$ and $k\in\mathbb{N}_{0}$, we introduce the space $\mathcal{P}^{k}(T):=\{z: T\to \mathbb{R} ; z \text{ is a polynomial of degree}\leq k\}$. Set
\begin{equation*}
\mathcal{P}^{k}(\mathscr{T}):=\{z\in L^2(\Omega) : z|_{T}\in \mathcal{P}^{k}(T) ~ \forall T\in\mathscr{T}\}.
\end{equation*}
Moreover, define the spaces
\begin{equation*}
\mathcal{P}_{\text{grad}}^{k+1}(\mathscr{T}):=\mathcal{P}^{k+1}(\mathscr{T})\cap H_0^1(\Omega), \qquad
\mathcal{P}_{\text{div}}^{k}(\mathscr{T}):=\mathbf{RT}^k(\mathscr{T})\cap H(\text{div},\Omega),
\end{equation*}
where $\mathbf{RT}^{k}(\mathscr{T})$ denotes the Raviart--Thomas finite element space of order $k$~\cite{MR0483555}. In addition, we define the spaces
\begin{equation*}
\mathcal{P}_{c,0}^{k+1}(\partial\mathscr{T}):=\text{tr}^{\text{grad}}(\mathcal{P}_{\text{grad}}^{k+1}(\mathscr{T})), \qquad \mathcal{P}^{k}(\partial\mathscr{T}):=\text{tr}^{\text{div}}(\mathcal{P}_{\text{div}}^{k}(\mathscr{T})).
\end{equation*}

We let $\mathcal{E}$ be the set of all mesh sides ($d$-faces) and let $\mathcal{E}_{\text{int}}$ be the set of all interior sides of $\mathscr{T}$. We denote the boundary elements of $T\in \mathscr{T}$ by $\mathcal{E}_{T}\subset \mathcal{E}$. Hence, for $\gamma\in\mathcal{E}_{T}$ with $T\in\mathscr{T}$, we define the spaces $\mathcal{P}^{k}(\gamma):=\{z: \gamma\to \mathbb{R} ; z \text{  is a polynomial of degree}\leq k\}$ and $\mathcal{P}^{k}(\mathcal{E}_{T}):=\{z\in L^2(\partial T) : z|_{\gamma}\in \mathcal{P}^{k}(\gamma) ~ \forall \gamma\in\mathcal{E}_{T}\}$.  

Given a discrete function $\boldsymbol\tau \in \mathcal{P}^{k}(\mathscr{T})^{d}$ with $k\in\mathbb{N}_{0}$, we define, for any internal side $\gamma \in \mathcal{E}_{\text{int}}$, the \emph{jump or interelement residual} 
\[
\llbracket \boldsymbol\tau\cdot \mathbf{n} \rrbracket:= \mathbf{n}^{+} \cdot \boldsymbol\tau|^{}_{T^{+}} + \mathbf{n}^{-} \cdot \boldsymbol\tau|^{}_{T^{-}},
\]
where $\mathbf{n} ^{+}, \mathbf{n} ^{-}$ denote the unit normals to $\gamma$ pointing outwards $T^{+}$, $T^{-} \in \mathscr{T}$, respectively; $T^{+}$, $T^{-} \in \mathscr{T}$ are such that $T^{+} \neq T^{-}$ and $\partial T^{+} \cap \partial T^{-} = \gamma$. If $\gamma\in\mathcal{E}\setminus\mathcal{E}_{\text{int}}$ is an exterior side on $\partial T$, for some $T\in\mathscr{T}$, then $\mathbf{n}$ corresponds to the outward unit normal on $\partial\Omega$ and we simply set $\llbracket \boldsymbol\tau\cdot \mathbf{n} \rrbracket:= \mathbf{n} \cdot \boldsymbol\tau|^{}_{T}$. Similarly, for any $z\in \mathcal{P}^{k}(\mathscr{T})$ that may be discontinuous across an interface $\gamma\in\mathcal{E}_{\text{int}}$, we define $\llbracket z\mathbf{n}\rrbracket := \mathbf{n}^{+} z|^{}_{T^{+}} + \mathbf{n}^{-}z|^{}_{T^{-}}$ and set $\llbracket z\mathbf{n}\rrbracket = z|_{T}\mathbf{n}$ on boundary sides $\gamma\in\mathcal{E}\setminus\mathcal{E}_{\text{int}}$.

Finally,  $\pi_{k}$ denotes the $L^2(\Omega)$-orthogonal projection onto $\mathcal{P}^k(\mathscr{T})$.


\section{The optimal control problem}\label{sec:the_OCP}

In this section we describe a general framework for linear quadratic optimal control problems with control constraints based on the DPG method.


\subsection{Problem formulation}

In order to present a general framework we introduce the following ingredients. Let $U,Y,V$, and $\mathcal{Y}$ be real Hilbert spaces. Let $U_{ad}$ be a non-empty, bounded, closed and convex subset of $U$. We consider a bounded bilinear form $b: Y\times V \to \mathbb{R}$ and, associated to it, the operator $B:Y\to V'$ defined by $\langle By,v\rangle_{V',V}=b(y,v)$. We also consider an \emph{observation} operator $O: Y \to \mathcal{Y}$ and a \emph{data} operator $D:U \to V'$; both operators being linear and bounded.

Given $\mathsf{y}_{d}\in \mathcal{Y}$ and a regularization parameter $\alpha>0$, we define the functional
\begin{equation*}
J(\mathsf{y},\mathsf{u}):=\frac{1}{2}\|O\mathsf{y} - \mathsf{y}_{d}\|_{\mathcal{Y}}^2 + \frac{\alpha}{2}\|\mathsf{u}\|_{U}^2.
\end{equation*}
The considered optimal control problem reads as follows:
\begin{equation}\label{def:opt_cont_prob}
\min\{J(\mathsf{y},\mathsf{u}) : (\mathsf{y},\mathsf{u})\in Y\times U_{ad}\} ~ \text{ subject to } ~ b(\mathsf{y},v) = \langle D\mathsf{u},v\rangle_{V',V} \quad \forall v\in V.
\end{equation}
We immediately note that the state equation appearing in \eqref{def:opt_cont_prob} can be equivalently rewritten in terms of operators as $B\mathsf{y} = D\mathsf{u}$ in $V'$.


\subsection{Existence of solution}

To study existence of an optimal solution for problem \eqref{def:opt_cont_prob}, we introduce the space $W=\{v\in V : b(y,v)=0  \text{ for all } y\in Y\}$. 
Suppose that the following $\inf$--$\sup$ conditions for bilinear form $b$ hold true: There exists a constant $C_{b}>0$ such that
\begin{equation}\label{eq:inf_sup_b}
  C_{b}\|y\|_{Y} \leq \sup_{v\in V}\frac{b(y,v)}{\|v\|_{V}} \quad \forall y\in Y \qquad\text{and}\qquad W=\{0\}.
\end{equation}
It is well known that, if \eqref{eq:inf_sup_b} holds, then the operator $B$ is an isomorphism \cite[Theorem 3.3]{MR163054}.
Thus, the control-to-state mapping $S:U\to Y$ given by $S\mathsf{u}:=B^{-1}D\mathsf{u}=\mathsf{y}$ is well defined.
In the next result we show existence of a unique optimal solution.

\begin{theorem}[existence of optimal solution]\label{thm:existence}
Assume that \eqref{eq:inf_sup_b} holds. Then, the optimal control problem \eqref{def:opt_cont_prob} admits a unique solution $(\bar{\mathsf{y}},\bar{\mathsf{u}})\in Y\times U_{ad}$.
\end{theorem}
\begin{proof}
The reduced cost functional $j(\mathsf{u}):=J(S\mathsf{u},\mathsf{u})$ is convex and continuous. Since $U_{ad}$ is non-empty, bounded, closed and convex, we thus invoke standard arguments to conclude the existence of an optimal solution $(\bar{\mathsf{y}},\bar{\mathsf{u}})$; see, e.g., \cite[Chapter II, Theorem 1.2]{MR0271512} and \cite[Theorem 2.14]{MR2583281}. The uniqueness follows since $\alpha > 0$.
\end{proof}


\subsection{Optimality conditions}

We characterize the optimal control $\bar{\mathsf{u}}$ through the next first order optimality condition: $\bar{\mathsf{u}}\in U_{ad}$ uniquely solves problem \eqref{def:opt_cont_prob} if and only if $j'(\bar{\mathsf{u}})(\mathsf{u} - \bar{\mathsf{u}}) \geq 0$ for all $\mathsf{u}\in U_{ad}$ \cite[Lemma 2.21]{MR2583281}. Here, $j(\bar{\mathsf{u}})=J(S\bar{\mathsf{u}},\bar{\mathsf{u}})$ and $j'(\cdot)$ denotes the Gate\^aux derivative of $j(\cdot)$.

In order to rewrite the inequality $j'(\bar{\mathsf{u}})(\mathsf{u} - \bar{\mathsf{u}}) \geq 0$ in a more practical form, we introduce, given $\mathsf{u}\in U$, the \emph{adjoint state} $\mathsf{p}\in V$ as the unique solution to
\begin{equation}\label{eq:adjoint_eq}
B^{*}\mathsf{p} = O^{*}\mathcal{I}_{\mathcal{Y}}^{\mathcal{Y}'}(O\mathsf{y} - \mathsf{y}_{d}) ~\text{ in } ~ Y',
\end{equation}
where $\mathsf{y}=S\mathsf{u}$. The weak form of \eqref{eq:adjoint_eq} is: Find $(\lambda,\mathsf{p})\in Y\times V$ such that
\begin{equation}\label{eq:adj_eq_mixed}
(\mathsf{p},v)_{V} - b(\lambda,v) = 0, \quad 
b(w,\mathsf{p})  =  \langle O^{*}\mathcal{I}_{\mathcal{Y}}^{\mathcal{Y}'}(O\mathsf{y} - \mathsf{y}_{d}),w\rangle_{Y',Y} \quad \forall (v,w)\in V\times Y.
\end{equation}
We immediately note that the existence and uniqueness of $(\lambda,\mathsf{p})\in Y\times V$ also follow from the inf-sup condition \eqref{eq:inf_sup_b}. We stress that the second equation of \eqref{eq:adj_eq_mixed} is equivalent to \eqref{eq:adjoint_eq}, while the first equation is equivalent to $\mathcal{I}_{V}^{V'}\mathsf{p} = B\lambda$, i.e., $\lambda =  B^{-1}\mathcal{I}_{V}^{V'}\mathsf{p}$. The main reason to consider the weak form as in \eqref{eq:adj_eq_mixed} is that its discrete stability is guaranteed solely by a discrete inf-sup condition, which is relatively easy to obtain in practice through the construction of suitable Fortin operators \cite[Section 2.2]{MR4090394}.

We are now in position to provide first order optimality conditions.

\begin{theorem}[first order optimality conditions]\label{thm:1st_order_opt_cond}
Assume that \eqref{eq:inf_sup_b} holds. Let $(\bar{\mathsf{y}},\bar{\mathsf{u}})\in Y\times U_{ad}$ be the unique solution to problem \eqref{def:opt_cont_prob}. Then, the optimal control $\bar{\mathsf{u}}\in U_{ad}$ satisfies the inequality
\begin{equation}\label{eq:var_ineq}
(\mathcal{I}_{U'}^{U}D^{*}\bar{\mathsf{p}} + \alpha \bar{\mathsf{u}}, \mathsf{u} - \bar{\mathsf{u}})_{U} \geq 0 \quad \forall \mathsf{u}\in U_{ad},
\end{equation}
where $\bar{\mathsf{p}}\in V$ denotes the unique solution to \eqref{eq:adjoint_eq} \textnormal{(}equivalently \eqref{eq:adj_eq_mixed}\textnormal{)} with $\mathsf{y}=\bar{\mathsf{y}}=S\bar{\mathsf{u}}$.
\end{theorem}
\begin{proof}
A simple computation shows that we can rewrite, for all $\mathsf{u}\in U_{ad}$, the inequality $j'(\bar{\mathsf{u}})(\mathsf{u} - \bar{\mathsf{u}}) \geq 0$ as follows
\begin{equation}\label{eq:var_ineq_I}
(O(\mathsf{y}-\bar{\mathsf{y}}),O\bar{\mathsf{y}}-\mathsf{y}_{d})_{\mathcal{Y}} + \alpha(\bar{\mathsf{u}},\mathsf{u} - \bar{\mathsf{u}})_{U} \geq 0.
\end{equation}
Let us concentrate on the first term on the left-hand side of \eqref{eq:var_ineq_I}. Invoke the adjoint equation \eqref{eq:adjoint_eq} and the fact that $B(\mathsf{y}-\bar{\mathsf{y}})=D(\mathsf{u}-\bar{\mathsf{u}})$ in $V'$, to obtain
\begin{multline*}
(O(\mathsf{y}-\bar{\mathsf{y}}),O\bar{\mathsf{y}}-\mathsf{y}_{d})_{\mathcal{Y}} 
=
\langle\mathsf{y}-\bar{\mathsf{y}},O^{*}\mathcal{I}_{\mathcal{Y}}^{\mathcal{Y}'}(O\bar{\mathsf{y}} - \mathsf{y}_{d})\rangle_{Y,Y'}
= 
\langle\mathsf{y}-\bar{\mathsf{y}},B^{*}\bar{\mathsf{p}}\rangle_{Y,Y'} \\
=
\langle B(\mathsf{y}-\bar{\mathsf{y}}),\bar{\mathsf{p}}\rangle_{V',V} 
=
\langle D(\mathsf{u}-\bar{\mathsf{u}}),\bar{\mathsf{p}}\rangle_{V',V} 
=
\langle \mathsf{u}-\bar{\mathsf{u}},D^{*}\bar{\mathsf{p}}\rangle_{U,U'} 
=
(\mathsf{u}-\bar{\mathsf{u}},\mathcal{I}_{U'}^{U}D^{*}\bar{\mathsf{p}})_{U}.
\end{multline*}
Using $(O(\mathsf{y}-\bar{\mathsf{y}}),O\bar{\mathsf{y}}-\mathsf{y}_{d})_{\mathcal{Y}}=
(\mathsf{u}-\bar{\mathsf{u}},\mathcal{I}_{U'}^{U}D^{*}\bar{\mathsf{p}})_{U}$ in \eqref{eq:var_ineq_I} concludes the proof.
\end{proof}


\section{Discrete approximation}\label{sec:discrete_approx}

In this section we analyze the practical DPG method for the optimal control problem \eqref{def:opt_cont_prob}. To accomplish this task we introduce the following ingredients. Given a finite dimensional subspace $V_{h} \subseteq V$, we define the discrete trial-to-test operator $\Theta_{h}: Y \to V_{h}$ by
\begin{equation*}
(\Theta_{h}\mathsf{y}, v_{h})_{V} = b(\mathsf{y}, v_{h}) \quad \forall v_{h}\in V_{h},~ \mathsf{y}\in Y.
\end{equation*}
Let $Y_{h}\subseteq Y$ be some finite dimensional subspace. We define the discrete optimal test space $V_{h}^{\Theta_{h}}:=\Theta_{h}(Y_{h})$. Additionally, let $U_{h}\subseteq U$ denote a -not necessarily discrete- subspace with $U_{ad,h}:=U_{h}\cap U_{ad} \neq \emptyset$. We immediately mention that, if the variational discretization approach \cite{MR2122182} is considered, then $U_{h}=U$.

Finally, we recall the definition of a Fortin operator.
\begin{definition}[Fortin operator]\label{def:fortin_op} 
Let $\Pi: V \to V_{h}$. We say that $\Pi$ is a \emph{Fortin} operator if there exists $C_{\Pi}> 0$ such that
\begin{equation*}
\|\Pi v\|_{V} \leq C_{\Pi}\|v\|_{V} \quad \textnormal{ and } \quad b(\mathsf{y}_{h}, v - \Pi v) = 0 \quad \forall \mathsf{y}_{h}\in Y_{h}, ~ \forall v\in V.
\end{equation*}
\end{definition}


\subsection{Discrete optimal control problem}\label{sec:discrete_ocp}

The fully discrete DPG method for optimal control problem \eqref{def:opt_cont_prob} reads as follows:
\begin{equation}\label{def:discrete_opt_cont_prob_i}
\min\{J(\mathsf{y}_{h},\mathsf{u}_{h}) : (\mathsf{y}_{h},\mathsf{u}_{h})\in Y_{h}\times U_{ad,h}\}
\end{equation}
subject to the discrete state equation
\begin{equation}\label{def:discrete_opt_cont_prob_ii}
b(\mathsf{y}_{h},v_{h}) = \langle D\mathsf{u}_{h},v_{h}\rangle_{V',V} \quad \forall v_{h}\in V_{h}^{\Theta_{h}}.
\end{equation}

\begin{remark}[equivalent formulation]
We notice that the discrete state equation \eqref{def:discrete_opt_cont_prob_ii} can be reformulated as follows: Find $(\mathsf{y}_{h},\varepsilon_{h})\in Y_{h}\times V_{h}$ such that
\begin{equation}\label{eq:equivalent_discrete_st_eq}
(\varepsilon_{h},v_{h})_{V} + b(\mathsf{y}_{h},v_{h}) = \langle D\mathsf{u}_{h},v_{h}\rangle_{V',V}, \quad b(w_{h},\varepsilon_{h})  =  0  \quad \forall (v_{h},w_{h})\in V_{h}\times Y_{h}.
\end{equation}
In what follows, we shall consider formulation \eqref{eq:equivalent_discrete_st_eq} for the discrete state equation.
\end{remark}

We now present the existence of a unique optimal discrete solution.

\begin{theorem}[existence of discrete solution]\label{thm:existence_discrete}
Assume that \eqref{eq:inf_sup_b} holds and that $\Pi: V \to V_{h}$ is a Fortin operator. Then, the discrete optimal control problem \eqref{def:discrete_opt_cont_prob_i}--\eqref{def:discrete_opt_cont_prob_ii} admits a unique solution $(\bar{\mathsf{y}}_{h},\bar{\mathsf{u}}_{h})\in Y_{h}\times U_{ad,h}$.
\end{theorem}
\begin{proof}
The assumptions immediately yield that problem \eqref{def:discrete_opt_cont_prob_ii}, equivalently \eqref{eq:equivalent_discrete_st_eq}, is well posed \cite[Theorem 2.1]{MR3143683}; see also \cite[Theorem 4.2]{MR3521055}. We can thus follow the arguments elaborated in the proof of Theorem \ref{thm:existence}. For brevity, we skip details.
\end{proof} 

As in the continuous case, we characterize the discrete optimal control $\bar{\mathsf{u}}_{h}$ through a first order optimality condition. To present it, we introduce the operator $S_{h}:U_{h} \to Y_{h}$ defined by $S_{h}\mathsf{u}_{h}=\mathsf{y}_{h}$, with $\mathsf{y}_{h}$ solution to \eqref{def:discrete_opt_cont_prob_ii}. The optimality condition reads: $\bar{\mathsf{u}}_{h}\in U_{ad,h}$ is optimal for \eqref{def:discrete_opt_cont_prob_i}--\eqref{def:discrete_opt_cont_prob_ii} if and only if it satisfies $j_{h}^{\prime}(\bar{\mathsf{u}}_{h})(\mathsf{u}_{h} - \bar{\mathsf{u}}_{h}) \geq 0$ for all $\mathsf{u}_{h}\in U_{ad,h}$, with $j_{h}(\bar{\mathsf{u}}_{h}):=J(S_{h}\bar{\mathsf{u}}_{h},\bar{\mathsf{u}}_{h})$. In order to rewrite this inequality, we introduce the discrete adjoint variable $(\lambda_{h},\mathsf{p}_{h})\in Y_{h}\times V_{h}$ as the unique solution to 
\begin{equation}\label{eq:discrete_adj_eq}
(\mathsf{p}_{h},v_{h})_{V} - b(\lambda_{h},v_{h}) = 0, \quad 
b(w_{h},\mathsf{p}_{h})  =  \langle O^{*}\mathcal{I}_{\mathcal{Y}}^{\mathcal{Y}'}(O\mathsf{y}_{h} - \mathsf{y}_{d}),w_{h}\rangle_{Y',Y} 
\end{equation}
for all $(v_{h},w_{h})\in V_{h}\times Y_{h}$.

We are now in position to provide discrete first order optimality conditions.

\begin{theorem}[discrete optimality conditions]\label{thm:discrete_1st_order_opt_cond}
Assume that \eqref{eq:inf_sup_b} holds and that $\Pi: V \to V_{h}$ is a Fortin operator. Let $(\bar{\mathsf{y}}_{h},\bar{\mathsf{u}}_{h})\in Y_{h}\times U_{ad,h}$ be the unique solution to problem \eqref{def:discrete_opt_cont_prob_i}--\eqref{def:discrete_opt_cont_prob_ii}. Then, the optimal control $\bar{\mathsf{u}}_{h}\in U_{ad,h}$ satisfies the inequality
\begin{equation}\label{eq:discrete_var_ineq}
(\mathcal{I}_{U'}^{U}D^{*}\bar{\mathsf{p}}_{h} + \alpha \bar{\mathsf{u}}_{h}, \mathsf{u}_{h} - \bar{\mathsf{u}}_{h})_{U} \geq 0 \quad \forall \mathsf{u}_{h}\in U_{ad,h},
\end{equation}
where $\bar{\mathsf{p}}_{h}\in V_{h}$ denotes the unique solution to \eqref{eq:discrete_adj_eq} with $\mathsf{y}_{h}=\bar{\mathsf{y}}_{h}$.
\end{theorem}
\begin{proof}
Observe that $j_{h}^{\prime}(\bar{\mathsf{u}}_{h})(\mathsf{u}_{h} - \bar{\mathsf{u}}_{h}) \geq 0$ for all $\mathsf{u}_{h}\in U_{ad,h}$ can be rewritten as
\begin{equation*}
(O(\mathsf{y}_{h}-\bar{\mathsf{y}}_{h}),O\bar{\mathsf{y}}_{h}-\mathsf{y}_{d})_{\mathcal{Y}} + \alpha(\bar{\mathsf{u}}_{h},\mathsf{u}_{h} - \bar{\mathsf{u}}_{h})_{U} \geq 0.
\end{equation*}
Let us concentrate on $(O(\mathsf{y}_{h}-\bar{\mathsf{y}}_{h}),O\bar{\mathsf{y}}_{h}-\mathsf{y}_{d})_{\mathcal{Y}}$. We notice that $(\mathsf{y}_{h} - \bar{\mathsf{y}}_{h},\varepsilon_{h} - \bar{\varepsilon}_{h})\in Y_{h}\times V_{h}$ solves the discrete problem
\begin{equation}\label{eq:aux_y-bary}
(\varepsilon_{h} - \bar{\varepsilon}_{h},v_{h})_{V} + b(\mathsf{y}_{h} - \bar{\mathsf{y}}_{h},v_{h}) = \langle D(\mathsf{u}_{h} - \bar{\mathsf{u}}_{h}),v_{h}\rangle_{V',V}, \quad b(w_{h},\varepsilon_{h} - \bar{\varepsilon}_{h})  =  0
\end{equation}
for all $(v_{h},w_{h})\in V_{h}\times Y_{h}$. We can thus use $(v_{h},w_{h})=(\bar{\mathsf{p}}_{h},\bar{\lambda}_{h})$ in \eqref{eq:aux_y-bary},  $(v_{h},w_{h})=(\varepsilon_{h} - \bar{\varepsilon}_{h},\mathsf{y}_{h} - \bar{\mathsf{y}}_{h})$ in \eqref{eq:discrete_adj_eq}, and the fact that $(\varepsilon_{h}-\bar{\varepsilon}_{h},\bar{\mathsf{p}}_{h})_{V}=0$. These arguments yield $(O(\mathsf{y}_{h}-\bar{\mathsf{y}}_{h}),O\bar{\mathsf{y}}_{h}-\mathsf{y}_{d})_{\mathcal{Y}} = (\mathsf{u}_{h}-\bar{\mathsf{u}}_{h},\mathcal{I}_{U'}^{U}D^{*}\bar{\mathsf{p}}_{h})_{U}$, which concludes the proof.
\end{proof}


\subsection{A priori error estimates}
We now present an approximation result for the solution of the optimal control problem \eqref{def:opt_cont_prob}.

\begin{theorem}[error estimates]\label{thm:a_priori_estimates}
Assume that \eqref{eq:inf_sup_b} holds and that $\Pi: V \to V_{h}$ is a Fortin operator. Let $(\bar{\mathsf{y}},\bar{\mathsf{u}})\in Y\times U_{ad}$ be the unique solution to problem \eqref{def:opt_cont_prob} and $(\bar{\mathsf{y}}_{h},\bar{\mathsf{u}}_{h})\in Y_{h}\times U_{ad,h}$ be the unique solution to problem \eqref{def:discrete_opt_cont_prob_i}--\eqref{def:discrete_opt_cont_prob_ii}. Then,
\begin{align}\label{eq:estimate_control_apriori}
\alpha^2\|\bar{\mathsf{u}} - \bar{\mathsf{u}}_{h}\|_{U}^2
\lesssim &
~\alpha\inf_{\mathsf{u}_{h}\in U_{ad,h}}(\mathcal{I}_{U'}^{U}D^{*}\bar{\mathsf{p}}_{h} + \alpha \bar{\mathsf{u}}_{h}, \mathsf{u}_{h} - \bar{\mathsf{u}})_{U}\\
&+ \inf_{w_{h}\in Y_{h}}\|\bar{\mathsf{y}} - w_{h}\|_{Y}^2 + \inf_{v_{h}\in V_{h}}\|\bar{\mathsf{p}} - v_{h}\|_{V}^2 + \inf_{w_{h}\in Y_{h}}\|\bar{\lambda} - w_{h}\|_{Y}^2, \nonumber \\
\|\bar{\mathsf{y}} - \bar{\mathsf{y}}_{h}\|_Y^2
\lesssim &
\inf_{w_{h}\in Y_{h}}\|\bar{\mathsf{y}} - w_{h}\|_{Y}^2 + \|\bar{\mathsf{u}} - \bar{\mathsf{u}}_{h}\|_{U}^2, \nonumber
\end{align}
and
\begin{equation*}
\|\bar{\mathsf{p}} - \bar{\mathsf{p}}_{h}\|_{V}^2 \lesssim \inf_{v_{h}\in V_{h}}\|\bar{\mathsf{p}} - v_{h}\|_{V}^2 + \inf_{w_{h}\in Y_{h}}\|\bar{\lambda} - w_{h}\|_{Y}^2 + \inf_{w_{h}\in Y_{h}}\|\bar{\mathsf{y}} - w_{h}\|_{Y}^2 + \|\bar{\mathsf{u}} - \bar{\mathsf{u}}_{h}\|_{U}^2,
\end{equation*}
where $(\bar{\lambda},\bar{\mathsf{p}})\in Y\times V$ and $(\bar{\lambda}_{h},\bar{\mathsf{p}}_{h})\in Y_{h}\times V_{h}$ denote the optimal adjoint state and its corresponding discrete approximation, respectively.
\end{theorem}
\begin{proof}
We begin by estimating the term $\|\bar{\mathsf{u}} - \bar{\mathsf{u}}_{h}\|_{U}$. Setting $\mathsf{u} = \bar{\mathsf{u}}_{h}$ in \eqref{eq:var_ineq}, adding and subtracting $\bar{\mathsf{u}}$ in \eqref{eq:discrete_var_ineq}, adding these two inequalities and reordering terms result in
\begin{equation}\label{eq:u_h-u_ineq}
\alpha\|\bar{\mathsf{u}} - \bar{\mathsf{u}}_{h}\|_{U}^2
\leq 
\inf_{\mathsf{u}_{h}\in U_{ad,h}}(\mathcal{I}_{U'}^{U}D^{*}\bar{\mathsf{p}}_{h} + \alpha \bar{\mathsf{u}}_{h}, \mathsf{u}_{h} - \bar{\mathsf{u}})_{U} + (\mathcal{I}_{U'}^{U}D^{*}[\bar{\mathsf{p}} - \bar{\mathsf{p}}_{h}], \bar{\mathsf{u}}_{h} - \bar{\mathsf{u}})_{U}.
\end{equation}
We now estimate $(\mathcal{I}_{U'}^{U}D^{*}[\bar{\mathsf{p}} - \bar{\mathsf{p}}_{h}], \bar{\mathsf{u}}_{h} - \bar{\mathsf{u}})_{U}$. To bound this term, we introduce the auxiliary variables $(\hat{\mathsf{y}}_{h},\hat{\varepsilon}_{h})\in Y_{h}\times V_{h}$ and $(\hat{\lambda}_{h},\hat{\mathsf{p}}_{h})\in Y_{h}\times V_{h}$ as the solutions to
\begin{equation}\label{eq:hat_y}
(\hat{\varepsilon}_{h},v_{h})_{V} + b(\hat{\mathsf{y}}_{h},v_{h})  =  \langle D\bar{\mathsf{u}},v_{h}\rangle_{V',V}, \quad
b(w_{h},\hat{\varepsilon}_{h}) =  0  \quad \forall (v_{h},w_{h})\in V_{h}\times Y_{h},
\end{equation}
and
\begin{equation}\label{eq:hat_p}
(\hat{\mathsf{p}}_{h},v_{h})_{V} - b(\hat{\lambda}_{h},v_{h})  =  0, \quad
b(w_{h},\hat{\mathsf{p}}_{h})  =  \langle O^{*}\mathcal{I}_{\mathcal{Y}}^{\mathcal{Y}'}(O\hat{\mathsf{y}}_{h} - \mathsf{y}_{d}),w_{h}\rangle_{Y',Y}  
\end{equation}
for all $(v_{h},w_{h})\in V_{h}\times Y_{h}$, respectively. Invoke the auxiliary solution $\hat{\mathsf{p}}_{h}$ and write
\begin{align}\label{eq:terms_I_II}
(\mathcal{I}_{U'}^{U}D^{*}[\bar{\mathsf{p}} - \bar{\mathsf{p}}_{h}], \bar{\mathsf{u}}_{h} - \bar{\mathsf{u}})_{U} 
=  & ~
(\mathcal{I}_{U'}^{U}D^{*}[\hat{\mathsf{p}}_{h} - \bar{\mathsf{p}}_{h}], \bar{\mathsf{u}}_{h} - \bar{\mathsf{u}})_{U} \\ & + (\mathcal{I}_{U'}^{U}D^{*}[\bar{\mathsf{p}} - \hat{\mathsf{p}}_{h}], \bar{\mathsf{u}}_{h} - \bar{\mathsf{u}})_{U}
 =: \mathsf{I} + \mathsf{II}. \nonumber
\end{align}
Let us bound the term $\mathsf{I}$. Note that $(\hat{\lambda}_{h} - \bar{\lambda}_{h},\hat{\mathsf{p}}_{h} - \bar{\mathsf{p}}_{h})\in Y_{h}\times V_{h}$ solves 
\begin{align}\label{eq:hat_p-bar_p}
(\hat{\mathsf{p}}_{h} - \bar{\mathsf{p}}_{h},v_{h})_{V} - b(\hat{\lambda}_{h}-\bar{\lambda}_{h},v_{h})  &=  0, \\
b(w_{h},\hat{\mathsf{p}}_{h} - \bar{\mathsf{p}}_{h})  &=  \langle O^{*}\mathcal{I}_{\mathcal{Y}}^{\mathcal{Y}'}O(\hat{\mathsf{y}}_{h} - \bar{\mathsf{y}}_{h}),w_{h}\rangle_{Y',Y} \nonumber
\end{align}
for all $(v_{h},w_{h})\in V_{h}\times Y_{h}$. On the other hand, we note that $(\bar{\mathsf{y}}_{h} - \hat{\mathsf{y}}_{h},\bar{\varepsilon}_{h} - \hat{\varepsilon}_{h})\in Y_{h}\times V_{h}$ solves, for all $(v_{h},w_{h})\in V_{h}\times Y_{h}$, the discrete system 
\begin{equation}\label{eq:hat_y-bar_y}
(\bar{\varepsilon}_{h} - \hat{\varepsilon}_{h},v_{h})_{V} + b(\bar{\mathsf{y}}_{h} - \hat{\mathsf{y}}_{h},v_{h})  =  \langle D(\bar{\mathsf{u}}_{h} - \bar{\mathsf{u}}),v_{h}\rangle_{V',V}, \quad
b(w_{h},\bar{\varepsilon}_{h} - \hat{\varepsilon}_{h}) =  0.
\end{equation}
Hence, we conclude, replacing $(v_{h},w_{h})=(\bar{\varepsilon}_{h} - \hat{\varepsilon}_{h},\bar{\mathsf{y}}_{h} - \hat{\mathsf{y}}_{h})$ in \eqref{eq:hat_p-bar_p} and $(v_{h},w_{h})=(\hat{\mathsf{p}}_{h} - \bar{\mathsf{p}}_{h},\hat{\lambda}_{h} - \bar{\lambda}_{h})$ in \eqref{eq:hat_y-bar_y}, that $\mathsf{I} = -\|O(\hat{\mathsf{y}}_{h} - \bar{\mathsf{y}}_{h})\|_{\mathcal{Y}}^2 \leq 0$.

To control the term $\mathsf{II}$, we invoke Young's inequality to arrive at
\begin{equation}\label{eq:estimate_II_Young}
\mathsf{II} 
\leq 
\frac{1}{2\alpha}\|\mathcal{I}_{U'}^{U}D^{*}[\bar{\mathsf{p}} - \hat{\mathsf{p}}_{h}]\|_{U}^2 + \frac{\alpha}{2}\|\bar{\mathsf{u}}_{h} - \bar{\mathsf{u}}\|_{U}^2.
\end{equation}
To estimate the term $\|\mathcal{I}_{U'}^{U}D^{*}[\bar{\mathsf{p}} - \hat{\mathsf{p}}_{h}]\|_{U}^2$, we introduce the auxiliary variable $(\tilde{\lambda}_{h},\tilde{\mathsf{p}}_{h})\in Y_{h}\times V_{h}$ as the unique solution to
\begin{equation*}
(\tilde{\mathsf{p}}_{h},v_{h})_{V} - b(\tilde{\lambda}_{h},v_{h})  =  0, \quad
b(w_{h},\tilde{\mathsf{p}}_{h})  =  \langle O^{*}\mathcal{I}_{\mathcal{Y}}^{\mathcal{Y}'}(O\bar{\mathsf{y}} - \mathsf{y}_{d}),w_{h}\rangle_{Y',Y}  
\end{equation*}
for all $(v_{h},w_{h})\in V_{h}\times Y_{h}$. The triangle inequality yields
\begin{equation*}
\|\mathcal{I}_{U'}^{U}D^{*}[\bar{\mathsf{p}} - \hat{\mathsf{p}}_{h}]\|_{U}^2
\leq
2\|\mathcal{I}_{U'}^{U}D^{*}[\bar{\mathsf{p}} - \tilde{\mathsf{p}}_{h}]\|_{U}^2 + 2\|\mathcal{I}_{U}^{U}D^{*}[\tilde{\mathsf{p}}_{h} - \hat{\mathsf{p}}_{h}]\|_{U}^2 =: \mathsf{II}_{1} + \mathsf{II}_{2}.
\end{equation*}
By \cite[Theorem 3.1]{MR4090394} $(\tilde{\lambda}_{h},\tilde{\mathsf{p}}_{h})$ is a quasi-best approximation of $(\bar{\lambda},\bar{\mathsf{p}})$ on $Y_{h}\times V_{h}$ and we thus arrive at $\mathsf{II}_{1} \lesssim \|\bar{\mathsf{p}} - \tilde{\mathsf{p}}_{h}\|_{V}^2 \lesssim \inf_{v_{h}\in V_{h}}\|\bar{\mathsf{p}} - v_{h}\|_{V}^2 + \inf_{w_{h}\in Y_{h}}\|\bar{\lambda} - w_{h}\|_{Y}^2$. Moreover, a stability estimate of the discrete system
\begin{equation*}
 (\tilde{\mathsf{p}}_{h} - \hat{\mathsf{p}}_{h},v_{h})_{V} - b(\tilde{\lambda}_{h} - \hat{\lambda}_{h},v_{h})  =  0, \quad
b(w_{h},\tilde{\mathsf{p}}_{h} - \hat{\mathsf{p}}_{h})  =  \langle O^{*}\mathcal{I}_{\mathcal{Y}}^{\mathcal{Y}'}O(\bar{\mathsf{y}} - \hat{\mathsf{y}}_{h}),w_{h}\rangle_{Y',Y}  
\end{equation*}
for all $(v_{h},w_{h})\in V_{h}\times Y_{h}$, implies that 
\begin{equation*}
\mathsf{II}_{2} \lesssim \|\tilde{\mathsf{p}}_{h} - \hat{\mathsf{p}}_{h}\|_{V}^2 \lesssim \|O^{*}\mathcal{I}_{\mathcal{Y}}^{\mathcal{Y}'}O(\bar{\mathsf{y}} - \hat{\mathsf{y}}_{h})\|_{Y'}^2 \lesssim \|\bar{\mathsf{y}} - \hat{\mathsf{y}}_{h}\|_{Y}^2 \lesssim \inf_{w_{h}\in Y_{h}}\|\bar{\mathsf{y}} - w_{h}\|_{Y}^2,
\end{equation*}
where, in the last inequality, we have used that $\hat{\mathsf{y}}_{h}$ corresponds to a quasi-best approximation of $\bar{\mathsf{y}}$ on $Y_{h}$ \cite[Theorem 4.2]{MR3521055}. Therefore, in view of the estimates obtained for $\mathsf{II}_{1}$ and $\mathsf{II}_{2}$, it follows that
\begin{equation}\label{eq:estimate_bar_p-hat_p}
\|\mathcal{I}_{U'}^{U}D^{*}[\bar{\mathsf{p}} - \hat{\mathsf{p}}_{h}]\|_{U}^2
\lesssim
\inf_{w_{h}\in Y_{h}}\|\bar{\mathsf{y}} - w_{h}\|_{Y}^2
+ \inf_{v_{h}\in V_{h}}\|\bar{\mathsf{p}} - v_{h}\|_{V}^2 + \inf_{w_{h}\in Y_{h}}\|\bar{\lambda} - w_{h}\|_{Y}^2.
\end{equation}
Using \eqref{eq:estimate_bar_p-hat_p} in \eqref{eq:estimate_II_Young}, and combining with \eqref{eq:terms_I_II} we conclude that
\begin{multline*}
(\mathcal{I}_{U'}^{U}D^{*}[\bar{\mathsf{p}} - \bar{\mathsf{p}}_{h}], \bar{\mathsf{u}}_{h} - \bar{\mathsf{u}})_{U}\\
\leq \frac{C}{\alpha}\left(\inf_{w_{h}\in Y_{h}}\|\bar{\mathsf{y}} - w_{h}\|_{Y}^2 + \inf_{v_{h}\in V_{h}}\|\bar{\mathsf{p}} - v_{h}\|_{V}^2 + \inf_{w_{h}\in Y_{h}}\|\bar{\lambda} - w_{h}\|_{Y}^2\right) + \frac{\alpha}{2}\|\bar{\mathsf{u}}_{h} - \bar{\mathsf{u}}\|_{U}^2,
\end{multline*}
with $C>0$, upon using the fact that $\mathsf{I}\leq 0$. We obtain the desired estimate for $\|\bar{\mathsf{u}} - \bar{\mathsf{u}}_{h}\|_{U}$ by using the latter bound in \eqref{eq:u_h-u_ineq}.

Let us now estimate the error $\|\bar{\mathsf{y}} - \bar{\mathsf{y}}_{h}\|_{Y}$. To accomplish this task, we invoke the auxiliary term $\hat{\mathsf{y}}_{h}$, defined in \eqref{eq:hat_y}, and the triangle inequality to arrive at 
\begin{equation*}
\|\bar{\mathsf{y}} - \bar{\mathsf{y}}_{h}\|_{Y} \leq
\|\bar{\mathsf{y}} - \hat{\mathsf{y}}_{h}\|_{Y} + \|\hat{\mathsf{y}}_{h} - \bar{\mathsf{y}}_{h}\|_{Y}.
\end{equation*}
The stability estimate $\|\hat{\mathsf{y}}_{h} - \bar{\mathsf{y}}_{h}\|_{Y}\lesssim \|\bar{\mathsf{u}} - \bar{\mathsf{u}}_{h}\|_{U}$ and the fact that $\hat{\mathsf{y}}_{h}$ corresponds to a quasi-best approximation of $\bar{\mathsf{y}}$ on $Y_{h}$ \cite[Theorem 4.2]{MR3521055}, prove
\begin{equation}\label{eq:y-y_h}
\|\bar{\mathsf{y}} - \bar{\mathsf{y}}_{h}\|_{Y}^2 \lesssim
\inf_{w_{h}\in Y_{h}}\|\bar{\mathsf{y}} - w_{h}\|_{Y}^2 + \|\bar{\mathsf{u}} - \bar{\mathsf{u}}_{h}\|_{U}^2.
\end{equation}

Bounding the term $\|\bar{\mathsf{p}} - \bar{\mathsf{p}}_{h}\|_{V}$ follows similar arguments as the ones that led to \eqref{eq:y-y_h}. For brevity, we skip those details.
\end{proof}

\begin{remark}[variational discretization]
When the variational discretization approach \cite{MR2122182} is considered, we have that $\inf_{\mathsf{u}_{h}\in U_{ad,h}}(\mathcal{I}_{U'}^{U}D^{*}\bar{\mathsf{p}}_{h} + \alpha \bar{\mathsf{u}}_{h}, \mathsf{u}_{h} - \bar{\mathsf{u}})_{U}\leq 0$. 
\end{remark}

\begin{remark}[estimation of $\|\bar{\mathsf{u}} - \bar{\mathsf{u}}_{h}\|_{U}^2$]\label{remark:estimation_eu}
If $\pi_{U_{h}}\bar{\mathsf{u}}\in U_{ad,h}$, with $\pi_{U_{h}}: U \to U_{h}$ being the orthogonal projection operator, we can obtain the estimate
\begin{equation*}
\inf_{\mathsf{u}_{h}\in U_{ad,h}}(\mathcal{I}_{U'}^{U}D^{*}\bar{\mathsf{p}}_{h} + \alpha \bar{\mathsf{u}}_{h}, \mathsf{u}_{h} - \bar{\mathsf{u}})_{U}  \lesssim \|(1-\pi_{U_{h}})\mathcal{I}_{U'}^{U}D^{*}\bar{\mathsf{p}}_{h}\|_{U}^{2} + \|(1-\pi_{U_{h}})\bar{\mathsf{u}}\|_{U}^2,
\end{equation*}
which yields, in view of the estimate \eqref{eq:estimate_control_apriori}, the a priori error bound
\begin{multline}\label{eq:sharp_estimate_control}
\|\bar{\mathsf{u}} - \bar{\mathsf{u}}_{h}\|_{U}^2
\lesssim 
\inf_{w_{h}\in Y_{h}}\|\bar{\mathsf{y}} - w_{h}\|_{Y}^2 + \inf_{v_{h}\in V_{h}}\|\bar{\mathsf{p}} - v_{h}\|_{V}^2  + \inf_{w_{h}\in Y_{h}}\|\bar{\lambda} - w_{h}\|_{Y}^2 \\ + \|(1-\pi_{U_{h}})\mathcal{I}_{U'}^{U}D^{*}\bar{\mathsf{p}}_{h}\|_{U}^{2} + \|(1-\pi_{U_{h}})\bar{\mathsf{u}}\|_{U}^2.
\end{multline}
\end{remark}


\subsection{A posteriori error estimates}
In this section, we present residual-type a posteriori error estimators associated to the optimal control problem \eqref{def:opt_cont_prob} and provide reliability and efficiency estimates.

Let us assume that there exists a computable control $\tilde{\mathsf{u}}\in U_{ad}$ such that 
\begin{equation}\label{eq:ineq_tilde_u}
(\mathcal{I}_{U'}^{U}D^{*}\bar{\mathsf{p}}_{h} + \alpha \tilde{\mathsf{u}}, \mathsf{u} - \tilde{\mathsf{u}})_{U} \geq 0\quad \forall\mathsf{u}\in U_{ad}.
\end{equation}
With $\tilde{\mathsf{u}}$ at hand, we introduce the variables $\tilde{\mathsf{y}}\in Y$ and $\tilde{\mathsf{p}}\in V$ as the solutions to
\begin{equation}\label{eq:tilde_y_and_tilde_p}
B\tilde{\mathsf{y}} = D\tilde{\mathsf{u}}\text{ in } V' \quad \text{ and }\quad B^{*}\tilde{\mathsf{p}} = O^{*}\mathcal{I}_{\mathcal{Y}}^{\mathcal{Y}'} (O\tilde{\mathsf{y}} - \mathsf{y}_{d}) ~\text{ in } ~ Y',
\end{equation}
respectively.

The next result establishes reliability estimates for the optimal control problem.

\begin{theorem}[reliability estimates]\label{thm:reliability_estimates}
Assume that \eqref{eq:inf_sup_b} holds and that $\Pi: V \to V_{h}$ is a Fortin operator. Let $(\bar{\mathsf{y}},\bar{\mathsf{u}})\in Y\times U_{ad}$ be the unique solution to problem \eqref{def:opt_cont_prob} and $(\bar{\mathsf{y}}_{h},\bar{\mathsf{u}}_{h})\in Y_{h}\times U_{ad,h}$ be the unique solution to problem \eqref{def:discrete_opt_cont_prob_i}--\eqref{def:discrete_opt_cont_prob_ii}. Then, with $\tilde{\mathsf{u}}\in U_{ad}$ from \eqref{eq:ineq_tilde_u}, we have that
\begin{align*}
\alpha^2\|\bar{\mathsf{u}} - \bar{\mathsf{u}}_{h}\|_{U}^2
&\!\lesssim \!
(1 \!+\!\alpha^2)\|\tilde{\mathsf{u}} - \bar{\mathsf{u}}_{h}\|_{U}^2 + \| D\bar{\mathsf{u}}_{h} - B\bar{\mathsf{y}}_{h}\|_{V'}^2 + \|O^{*}\mathcal{I}_{\mathcal{Y}}^{\mathcal{Y}'}(O\bar{\mathsf{y}}_{h} - \mathsf{y}_{d}) - B^{*}\bar{\mathsf{p}}_{h}\|_{Y'}^2, \\
\|\bar{\mathsf{y}} - \bar{\mathsf{y}}_{h}\|_Y^2
& \!\lesssim \!
\| D\bar{\mathsf{u}}_{h} - B\bar{\mathsf{y}}_{h}\|_{V'}^2 + \|\bar{\mathsf{u}} - \bar{\mathsf{u}}_{h}\|_{U}^2, 
\end{align*}
and
\begin{equation*}
\|\bar{\mathsf{p}} - \bar{\mathsf{p}}_{h}\|_{V}^2 \lesssim \| D\bar{\mathsf{u}}_{h} - B\bar{\mathsf{y}}_{h}\|_{V'}^2 + \|O^{*}\mathcal{I}_{\mathcal{Y}}^{\mathcal{Y}'}(O\bar{\mathsf{y}}_{h} - \mathsf{y}_{d}) - B^{*}\bar{\mathsf{p}}_{h}\|_{Y'}^2 + \|\bar{\mathsf{u}} - \bar{\mathsf{u}}_{h}\|_{U}^2,
\end{equation*}
where $(\bar{\lambda},\bar{\mathsf{p}})\in Y\times V$ and $(\bar{\lambda}_{h},\bar{\mathsf{p}}_{h})\in Y_{h}\times V_{h}$ denote the optimal adjoint state and its corresponding discrete approximation, respectively.
\end{theorem}
\begin{proof}
The proof is divided into three steps.

\emph{Step 1.} (estimate of $\|\bar{\mathsf{u}} - \bar{\mathsf{u}}_{h}\|_{U}$). Invoke the auxiliary variable $\tilde{\mathsf{u}}\in U_{ad}$ and use the triangle inequality to see that
\begin{equation*}
\alpha^2\|\bar{\mathsf{u}} - \bar{\mathsf{u}}_{h}\|_{U}^2 
\leq
2\alpha^2\|\bar{\mathsf{u}} - \tilde{\mathsf{u}}\|_{U}^2  + 2\alpha^2\|\tilde{\mathsf{u}} - \bar{\mathsf{u}}_{h}\|_{U}^2.
\end{equation*}
We first concentrate on $\|\bar{\mathsf{u}} - \tilde{\mathsf{u}}\|_{U}^2$. Set $\mathsf{u}=\tilde{\mathsf{u}}$ in \eqref{eq:var_ineq}, $\mathsf{u}=\bar{\mathsf{u}}$ in \eqref{eq:ineq_tilde_u}, add the obtained inequalities, and use the auxiliary variable $\tilde{\mathsf{p}}\in V$ (defined in \eqref{eq:tilde_y_and_tilde_p}) to arrive at
\begin{align}\label{eq:u_h-u_ineq_a_posteriori}
\alpha\|\bar{\mathsf{u}} - \tilde{\mathsf{u}}\|_{U}^2
& \leq 
(\mathcal{I}_{U'}^{U}D^{*}[\bar{\mathsf{p}} - \bar{\mathsf{p}}_{h}], \tilde{\mathsf{u}} - \bar{\mathsf{u}})_{U} \\
&= (\mathcal{I}_{U'}^{U}D^{*}[\bar{\mathsf{p}} - \tilde{\mathsf{p}}], \tilde{\mathsf{u}} - \bar{\mathsf{u}})_{U} + (\mathcal{I}_{U'}^{U}D^{*}[\tilde{\mathsf{p}} - \bar{\mathsf{p}}_{h}], \tilde{\mathsf{u}} - \bar{\mathsf{u}})_{U}. \nonumber
\end{align}
To estimate the term $(\mathcal{I}_{U'}^{U}D^{*}[\bar{\mathsf{p}} - \tilde{\mathsf{p}}], \tilde{\mathsf{u}} - \bar{\mathsf{u}})_{U}$ in \eqref{eq:u_h-u_ineq_a_posteriori}, we use the fact that $(\mathcal{I}_{U'}^{U}D^{*}[\bar{\mathsf{p}} - \tilde{\mathsf{p}}], \tilde{\mathsf{u}} - \bar{\mathsf{u}})_{U} = - \|O(\tilde{\mathsf{y}} - \bar{\mathsf{y}})\|_{\mathcal{Y}}^{2}\leq 0$. This and Young's inequality yield
\begin{equation*}
\alpha\|\bar{\mathsf{u}} - \tilde{\mathsf{u}}\|_{U}^2 
\leq 
(\mathcal{I}_{U'}^{U}D^{*}[\tilde{\mathsf{p}} - \bar{\mathsf{p}}_{h}], \tilde{\mathsf{u}} - \bar{\mathsf{u}})_{U} 
\leq 
\frac{1}{2\alpha}\|\mathcal{I}_{U'}^{U}D^{*}[\tilde{\mathsf{p}} - \bar{\mathsf{p}}_{h}]\|_{U}^2 + \frac{\alpha}{2}\|\bar{\mathsf{u}} - \tilde{\mathsf{u}}\|_{U}^2.
\end{equation*}
Consequently, we have that $\alpha^2\|\bar{\mathsf{u}} - \tilde{\mathsf{u}}\|_{U}^2 \leq \|\mathcal{I}_{U'}^{U}D^{*}[\tilde{\mathsf{p}} - \bar{\mathsf{p}}_{h}]\|_{U}^2$. To bound the term $\|\mathcal{I}_{U'}^{U}D^{*}[\tilde{\mathsf{p}} - \bar{\mathsf{p}}_{h}]\|_{U}^2$, we introduce the variable $\hat{\mathsf{p}}\in V$ as the unique solution to
\begin{equation}\label{eq:hat_p_a_posteriori}
B^{*}\hat{\mathsf{p}} = O^{*}\mathcal{I}_{\mathcal{Y}}^{\mathcal{Y}'} (O\bar{\mathsf{y}}_{h} - \mathsf{y}_{d}) ~\text{ in } ~ Y'.
\end{equation}
With this variable at hand, we use the triangle inequality to obtain that
\begin{equation}\label{eq:estimate_I_II_a_posteriori}
\|\mathcal{I}_{U'}^{U}D^{*}[\tilde{\mathsf{p}} - \bar{\mathsf{p}}_{h}]\|_{U}^2
\leq 
2\|\mathcal{I}_{U'}^{U}D^{*}[\hat{\mathsf{p}} - \bar{\mathsf{p}}_{h}]\|_{U}^2 + 2\|\mathcal{I}_{U'}^{U}D^{*}[\tilde{\mathsf{p}} - \hat{\mathsf{p}}]\|_{U}^2 =: \mathsf{I} + \mathsf{II}.
\end{equation}
To bound $\mathsf{I}$, we note that $\bar{\mathsf{p}}_{h}$ corresponds to the discrete approximation of $\hat{\mathsf{p}}$ in $V_{h}$. Therefore, \cite[Theorem 4.1]{MR4090394} immediately yields that
$\mathsf{I} \lesssim \|O^{*}\mathcal{I}_{\mathcal{Y}}^{\mathcal{Y}'}(O\bar{\mathsf{y}}_{h} - \mathsf{y}_{d}) - B^{*}\bar{\mathsf{p}}_{h}\|_{Y'}^2$. To control the term $\mathsf{II}$, we invoke a stability estimate for $\tilde{\mathsf{p}} - \hat{\mathsf{p}}\in V$. This gives the bound $\mathsf{II} \lesssim \|\tilde{\mathsf{p}} - \hat{\mathsf{p}}\|_{V}^2 \lesssim \|\tilde{\mathsf{y}} - \bar{\mathsf{y}}_{h}\|_{Y}^2$. To control the term $\|\tilde{\mathsf{y}} - \bar{\mathsf{y}}_{h}\|_{Y}^2$, we introduce $\hat{\mathsf{y}}\in Y$ as the unique solution to $B\hat{\mathsf{y}} = D\bar{\mathsf{u}}_{h}$ in $V'$, and use the triangle inequality to obtain
\begin{equation*}
\|\tilde{\mathsf{y}} - \bar{\mathsf{y}}_{h}\|_{Y}^2 
\lesssim
\|\tilde{\mathsf{y}} - \hat{\mathsf{y}}\|_{Y}^2 + \|\hat{\mathsf{y}} - \bar{\mathsf{y}}_{h}\|_{Y}^2 
\lesssim
\|\tilde{\mathsf{y}} - \hat{\mathsf{y}}\|_{Y}^2 + \| D\bar{\mathsf{u}}_{h} - B\bar{\mathsf{y}}_{h}\|_{V'}^2,
\end{equation*}
upon using the fact that $\bar{\mathsf{y}}_{h}\in Y_{h}$ corresponds to a quasi-best approximation of $\hat{y}$ and thus $\|\hat{\mathsf{y}} - \bar{\mathsf{y}}_{h}\|_{Y} \eqsim \|B\hat{\mathsf{y}} - B\bar{\mathsf{y}}_{h}\|_{V'} = \| D\bar{\mathsf{u}}_{h} - B\bar{\mathsf{y}}_{h}\|_{V'}$. Therefore, we invoke the stability estimate of the problem that $\tilde{\mathsf{y}} - \hat{\mathsf{y}}\in V$ solves to arrive at
\begin{equation*}
\mathsf{II} 
\lesssim \|\tilde{\mathsf{p}} - \hat{\mathsf{p}}\|_{V}^2 
\lesssim \|\tilde{\mathsf{y}} - \bar{\mathsf{y}}_{h}\|_{Y}^2
\lesssim \|\tilde{\mathsf{u}} - \bar{\mathsf{u}}_{h}\|_{U}^2 +  \| D\bar{\mathsf{u}}_{h} - B\bar{\mathsf{y}}_{h}\|_{V'}^2.
\end{equation*}
We conclude by using the estimates obtained for $\mathsf{I}$ and $\mathsf{II}$ in \eqref{eq:estimate_I_II_a_posteriori}, and using the resulting estimation in $\alpha^2\|\bar{\mathsf{u}} - \tilde{\mathsf{u}}\|_{U}^2 \leq \|\mathcal{I}_{U'}^{U}D^{*}[\tilde{\mathsf{p}} - \bar{\mathsf{p}}_{h}]\|_{U}^2$. This shows that
\begin{multline*}
\alpha^2\|\bar{\mathsf{u}} - \bar{\mathsf{u}}_{h}\|_{U}^2 
\leq
2\alpha^2\|\bar{\mathsf{u}} - \tilde{\mathsf{u}}\|_{U}^2  + 2\alpha^2\|\tilde{\mathsf{u}} - \bar{\mathsf{u}}_{h}\|_{U}^2\\
\lesssim
(1 + \alpha^2)\|\tilde{\mathsf{u}} - \bar{\mathsf{u}}_{h}\|_{U}^2 + \| D\bar{\mathsf{u}}_{h} - B\bar{\mathsf{y}}_{h}\|_{V'}^2 + \|O^{*}\mathcal{I}_{\mathcal{Y}}^{\mathcal{Y}'}(O\bar{\mathsf{y}}_{h} - \mathsf{y}_{d}) - B^{*}\bar{\mathsf{p}}_{h}\|_{Y'}^2.
\end{multline*}

\emph{Step 2.} (estimate of $\|\bar{\mathsf{y}} - \bar{\mathsf{y}}_{h}\|_{Y}$). We invoke the auxiliary term $\hat{\mathsf{y}}$, solution to $B\hat{\mathsf{y}} = D\bar{\mathsf{u}}_{h}$ in $V'$, in combination with the triangle inequality, to obtain
\begin{equation*}
\|\bar{\mathsf{y}} - \bar{\mathsf{y}}_{h}\|_{Y} 
\leq
\|\bar{\mathsf{y}} - \hat{\mathsf{y}}\|_{Y} + \|\hat{\mathsf{y}} - \bar{\mathsf{y}}_{h}\|_{Y}.
\end{equation*}
Hence, using the stability estimate $\|\bar{\mathsf{y}} - \hat{\mathsf{y}}\|_{Y}\lesssim \|\bar{\mathsf{u}} - \bar{\mathsf{u}}_{h}\|_{U}$ and that $\|\hat{\mathsf{y}} - \bar{\mathsf{y}}_{h}\|_{Y} \eqsim \| D\bar{\mathsf{u}}_{h} - B\bar{\mathsf{y}}_{h}\|_{V'}$ we conclude 
\begin{equation}\label{eq:y-y_h_aposteriori}
\|\bar{\mathsf{y}} - \bar{\mathsf{y}}_{h}\|_{Y}^2 \lesssim
\|D\bar{\mathsf{u}}_{h} - B\bar{\mathsf{y}}_{h}\|_{V'}^2 + \|\bar{\mathsf{u}} - \bar{\mathsf{u}}_{h}\|_{U}^2.
\end{equation}

\emph{Step 3.} (estimate of $\|\bar{\mathsf{p}} - \bar{\mathsf{p}}_{h}\|_{V}$). We follow similar arguments to the ones that led to \eqref{eq:y-y_h_aposteriori}. For brevity, we skip details.
\end{proof}

\begin{remark}[on the existence of $\tilde{\mathsf{u}}$]
Depending on the particular structure of the admissible set $U_{ad}$, it is possible to compute the auxiliary solution $\tilde{\mathsf{u}}$ by means of a projection operator and the discrete adjoint state $\bar{\mathsf{p}}_{h}$; see, e.g., section \ref{sec:a_posteriori_Poisson}.
\end{remark}

\begin{remark}[variational discretization]
When the variational discretization approach \cite{MR2122182} is considered, we have that $\tilde{\mathsf{u}}=\bar{\mathsf{u}}_{h}$ and thus $\|\tilde{\mathsf{u}} - \bar{\mathsf{u}}_{h}\|_{U} = 0$. 
\end{remark}

We now provide efficiency estimates.

\begin{theorem}[efficiency estimates]\label{thm:efficiency_estimates}
In the framework of Theorem \ref{thm:reliability_estimates}, we have 
\begin{equation*}
\|\tilde{\mathsf{u}} - \bar{\mathsf{u}}_{h}\|_{U}
\leq
\|\tilde{\mathsf{u}} - \bar{\mathsf{u}}\|_{U} + \|\bar{\mathsf{u}} - \bar{\mathsf{u}}_{h}\|_{U}, \qquad
\|D\bar{\mathsf{u}}_{h} - B\bar{\mathsf{y}}_{h}\|_{V'}
\lesssim 
\|\bar{\mathsf{y}} - \bar{\mathsf{y}}_{h}\|_{Y} + \|\bar{\mathsf{u}} - \bar{\mathsf{u}}_{h}\|_{U}, 
\end{equation*}
and
\begin{equation*}
\|O^{*}\mathcal{I}_{\mathcal{Y}}^{\mathcal{Y}'}(O\bar{\mathsf{y}}_{h} - \mathsf{y}_{d}) - B^{*}\bar{\mathsf{p}}_{h}\|_{Y'} 
\lesssim
\|\bar{\mathsf{p}} - \bar{\mathsf{p}}_{h}\|_{V} + \|\bar{\mathsf{y}} - \bar{\mathsf{y}}_{h}\|_{Y}.
\end{equation*}
\end{theorem}
\begin{proof}  
The estimate for $\|\tilde{\mathsf{u}} - \bar{\mathsf{u}}_{h}\|_{U}$ follows from the use of the triangle inequality. 

To estimate $\|D\bar{\mathsf{u}}_{h} - B\bar{\mathsf{y}}_{h}\|_{V'}$, we use the auxiliary variable $\hat{\mathsf{y}}\in Y$, defined as the unique solution to $B\hat{\mathsf{y}} = D\bar{\mathsf{u}}_{h}$ in $V'$, that $\|\hat{\mathsf{y}} - \bar{\mathsf{y}}_{h}\|_{Y} \eqsim \| D\bar{\mathsf{u}}_{h} - B\bar{\mathsf{y}}_{h}\|_{V'}$, and the stability estimate associated to the problem that $\hat{\mathsf{y}} - \bar{\mathsf{y}}\in Y$ solves. These ingredients yield
\begin{equation*}
\|D\bar{\mathsf{u}}_{h} - B\bar{\mathsf{y}}_{h}\|_{V'}
\lesssim 
\|\hat{\mathsf{y}} - \bar{\mathsf{y}}_{h}\|_{Y}
\leq 
\|\hat{\mathsf{y}} - \bar{\mathsf{y}}\|_{Y} + \|\bar{\mathsf{y}} - \bar{\mathsf{y}}_{h}\|_{Y}
\lesssim 
\|\bar{\mathsf{u}} - \bar{\mathsf{u}}_{h} \|_{U} + \|\bar{\mathsf{y}} - \bar{\mathsf{y}}_{h}\|_{Y}.
\end{equation*}
To control the last term, we invoke the auxiliary variable $\hat{\mathsf{p}}\in V$ (solution to \eqref{eq:hat_p_a_posteriori}), \cite[Theorem 4.1]{MR4090394}, and the triangle inequality to arrive at
\begin{equation*}
\|O^{*}\mathcal{I}_{\mathcal{Y}}^{\mathcal{Y}'}(O\bar{\mathsf{y}}_{h} - \mathsf{y}_{d}) - B^{*}\bar{\mathsf{p}}_{h}\|_{Y'} 
\lesssim
\|\hat{\mathsf{p}} - \bar{\mathsf{p}}_{h}\|_{V}
\leq
\|\hat{\mathsf{p}} - \bar{\mathsf{p}}\|_{V} + \|\bar{\mathsf{p}} - \bar{\mathsf{p}}_{h}\|_{V}.
\end{equation*}
This, in view of the stability estimate $\|\hat{\mathsf{p}} - \bar{\mathsf{p}}\|_{V}\lesssim \|\bar{\mathsf{y}}_{h} - \bar{\mathsf{y}}\|_{Y}$, allows us to conclude.
\end{proof}

\begin{remark}[computability of the error estimators]\label{remark:error_estimators}
We note that the terms $\|D\bar{\mathsf{u}}_{h} - B\bar{\mathsf{y}}_{h}\|_{V'}$ and $\|O^{*}\mathcal{I}_{\mathcal{Y}}^{\mathcal{Y}'}(O\bar{\mathsf{y}}_{h} - \mathsf{y}_{d}) - B^{*}\bar{\mathsf{p}}_{h}\|_{Y'}$ are not feasibly computable. For practical reasons it is thus necessary to provide equivalent error estimators that are easy to evaluate. In the examples below (section \ref{sec:examples}) these terms can be localized following ideas from \cite{MR4090394}.
\end{remark}


\section{Examples}\label{sec:examples}

In this section we provide some model problems where the framework developed in sections \ref{sec:the_OCP} and \ref{sec:discrete_approx} can be applied. Remarkable examples that are not treated here, but that involve similar functional settings for the state equation include convection-dominated diffusion, linear elasticity, and acoustics; see \cite[Section 4]{MR4090394}. Other examples that also fit our framework are the ones that involve \emph{finite-dimensional controls}: controls that are represented by finitely many real values. These controls are used in various real applications of optimal control theory, where it can be difficult to practically implement control functions that vary arbitrarily in space \cite{MR2440722}.

In what follows, we let $\Omega\subset \mathbb{R}^{d}$, with $d\geq 2$, and  $\mathscr{T}$ be given as in section \ref{sec:broken_Sob_spaces}. We mention that the approximation spaces will be piecewise polynomial with respect to the mesh $\mathscr{T}$ (see section \ref{sec:approximation_spaces}); convergence thus hinges on a sequence of discrete spaces and meshes. Consequently, the examples involve a mesh $\mathscr{T}$ which is not ﬁxed. 


\subsection{Poisson problem}\label{sec:Poisson_problem}
We study the control problem \eqref{def:opt_cont_prob} subject to the Poisson problem which is written as the first-order system,
\begin{equation*}
\text{div } \boldsymbol\sigma = \mathsf{u} \text{ in } \Omega, \quad \boldsymbol\sigma + \nabla\mathsf{y}_{o} = 0  \text{ in } \Omega, \quad \mathsf{y}_{o} = 0 \text{ on }\partial\Omega.
\end{equation*}

We consider the following spaces for the optimal control problem:
\[
U = \mathcal{Y} = L^2(\Omega), 
\quad Y = L^2(\Omega) \times L^2(\Omega)^d \times H_{0}^{1/2}(\partial\mathscr{T}) \times H^{-1/2}(\partial\mathscr{T}), 
\]
and $V = H^1(\mathscr{T}) \times H(\text{div},\mathscr{T})$; see section \ref{sec:broken_Sob_spaces}. We equip  $Y$ and $V$ with the norms
\begin{align*}
\|\mathsf{y}\|_{Y}^2 :&= \|\mathsf{y}_{o}\|_{L^2(\Omega)}^2 + \|\boldsymbol\sigma\|_{L^2(\Omega)^d}^2 + \|\hat{\mathsf{y}}_{o}\|_{1/2,\partial\mathscr{T}}^2 + \|\hat{\sigma}\|_{-1/2,\partial\mathscr{T}}^2, \quad \forall \mathsf{y}=\!(\mathsf{y}_{o},\boldsymbol\sigma,\hat{\mathsf{y}}_o,\hat{\sigma})\in Y, \\
\|v\|_{V}^2 :&= \|\mathsf{v}\|_{H^1(\mathscr{T})}^2 + \|\mathbf{v}\|_{H(\text{div},\mathscr{T})}^2, \quad  \forall v=\!(\mathsf{v},\mathbf{v})\in V.
\end{align*}
The admissible set $U_{ad}$ is defined by $U_{ad}:=\{u \in L^2(\Omega):  \texttt{a} \leq u(x) \leq \texttt{b} \text{ a.e. in } \Omega \}$, where the bounds $\texttt{a},\texttt{b} \in \mathbb{R}$ are such that $\texttt{a} < \texttt{b}$.

We define the bilinear form $b:Y\times V\to \mathbb{R}$ on the basis of an \emph{ultraweak formulation} of the Poisson equation \cite[Section 3.1]{MR2837484}:
\begin{equation}\label{def:bilinear_Poisson}
b(\mathsf{y},v) := (\boldsymbol\sigma, \mathbf{v} - \nabla \mathsf{v})_{\mathscr{T}} - (\mathsf{y}_{o}, \text{div }\mathbf{v})_{\mathscr{T}} +  \langle \hat{\mathsf{y}}_{o}, \mathbf{v}\cdot \mathbf{n}\rangle_{\partial\mathscr{T}} + \langle \mathsf{v}, \hat{\sigma}\rangle_{\partial\mathscr{T}}.
\end{equation}
Here, $\mathsf{y}=(\mathsf{y}_{o},\boldsymbol\sigma,\hat{\mathsf{y}}_o,\hat{\sigma})\in Y$, $v=(\mathsf{v},\mathbf{v})\in V$, $(\cdot,\cdot)_{\mathscr{T}}:=\sum_{T\in\mathscr{T}}(\cdot,\cdot)_{L^2(T)}$, and $\langle \cdot,\cdot\rangle_{\partial\mathscr{T}}:=\sum_{T\in\mathscr{T}}\langle \cdot,\cdot\rangle_{\partial T}$.

The observation operator $O$ is defined by $O\mathsf{y}=\mathsf{y}_{o}$, i.e., the mapping $O$ yields the first component of the vector $\mathsf{y}$. Finally, the mapping $D$ is defined as the continuous embedding $L^2(\Omega)\hookrightarrow L^2(\Omega)\times \{0\} \hookrightarrow V'$. Hence, $\langle D\mathsf{u},v\rangle_{V',V} = (\mathsf{u},\mathsf{v})_{L^2(\Omega)}$ and we have as state equation in \eqref{def:opt_cont_prob} the problem: $b(\mathsf{y},v) = (\mathsf{u},\mathsf{v})_{L^2(\Omega)}$ for all $v=(\mathsf{v},\mathbf{v})\in V$.

On the basis of the previous ingredients it follows, in view of \cite[Section 3]{MR3521055}, that the bilinear form \eqref{def:bilinear_Poisson} satisfies the inf-sup condition \eqref{eq:inf_sup_b}. Theorems \ref{thm:existence} and \ref{thm:1st_order_opt_cond} thus yield the existence and uniqueness of an optimal solution pair $(\bar{\mathsf{y}},\bar{\mathsf{u}})\in Y\times U_{ad}$ for problem \eqref{def:opt_cont_prob} satisfying, with $\bar{\lambda} := (\bar{\zeta},\bar{\boldsymbol\mu},\overline{\hat{\zeta}},\overline{\hat{\mu}})\in Y$ and $\bar{\mathsf{p}} :=(\bar{\mathsf{p}}_{o},\bar{\mathbf{p}}) \in V$, the first order optimality condition \eqref{eq:var_ineq}; the latter being written now as follows: $(\bar{\mathsf{p}}_{o} + \alpha\bar{\mathsf{u}}, \mathsf{u} - \bar{\mathsf{u}})_{L^2(\Omega)} \geq 0$ for all $\mathsf{u}\in U_{ad}$. In addition, this variational inequality yields a characterization of the optimal control $\bar{\mathsf{u}}$ \cite[Theorem 2.28]{MR2583281}:
\begin{equation*}
\bar{\mathsf{u}}(x) = \Pi_{[\texttt{a},\texttt{b}]}(-\alpha^{-1}\bar{\mathsf{p}}_{o}(x)):=\min\{\texttt{b},\max\{\texttt{a},-\alpha^{-1}\bar{\mathsf{p}}_{o}(x)\}\} \quad \text{for a.e. } x\in\Omega.
\end{equation*}
Here, $\Pi_{[\texttt{a},\texttt{b}]}: L^2(\Omega) \to U_{ad}$. We immediately notice that, since $\bar{\mathsf{p}}_{o}\in H^{1}(\mathscr{T})$, $\bar{\mathsf{u}} \in H^{1}(\mathscr{T})$; see, for instance, \cite[Theorem A.1]{MR1786735}.


\subsubsection{Discrete approximation}\label{sec:discrete_Poisson}
 Assume that each element $T\in\mathscr{T}$ is a simplex. We consider the following conforming discrete spaces (see section \ref{sec:approximation_spaces})
\begin{equation*}
Y_{h}:=\mathcal{P}^{k}(\mathscr{T})\times \mathcal{P}^{k}(\mathscr{T})^{d}\times \mathcal{P}^{k+1}_{c,0}(\partial\mathscr{T})\times \mathcal{P}^{k}(\partial\mathscr{T}), \quad V_{h}:=\mathcal{P}^{k_{1}}(\mathscr{T})\times\mathcal{P}^{k_{2}}(\mathscr{T})^{d},
\end{equation*}
$\mathcal{Y}_{h}:=\mathcal{P}^{k}(\mathscr{T})$, and $U_{h}:=\mathcal{P}^{k}(\mathscr{T})$, with $k\in\mathbb{N}_{0}$ and $k_{1},k_{2}\in\mathbb{N}$ satisfying $k_{1} \geq k + d +1$ and $k_{2} \geq k + 2$. We immediately mention that dim($V_{h}$) $\geq$ dim($Y_{h}$). A Fortin operator $\Pi\colon V\to V_h$ (see Definition \ref{def:fortin_op}) for the previous chosen discrete spaces exists.
For $v = (\mathsf{v},\mathbf{v})\in V$ we set $\Pi v = (\Pi^\nabla \mathsf{v},\Pi^{\mathrm{div}}\mathbf{v})$, where $\Pi^\nabla$ is the Fortin operator designed in~\cite[Section 3.1]{2023arXiv230113021F} and $\Pi^{\mathrm{div}}$ is the operator from~\cite[Lemma~3.3]{MR3143683}. 
Besides the statements of Definition \ref{def:fortin_op}, the operator $\Pi$ also satisfies
\begin{equation}\label{eq:Fortin_op_control}
\langle D\bar{\mathsf{u}}_{h},\Pi v \rangle_{V',V}= (\bar{\mathsf{u}}_{h},\Pi^{\nabla} \mathsf{v})_{L^2(\Omega)} = (\bar{\mathsf{u}}_{h},\mathsf{v})_{L^2(\Omega)}= \langle D\bar{\mathsf{u}}_{h}, v \rangle_{V',V} \quad \forall v\in V,
\end{equation}
which directly follows from~\cite[Eq. (7c)]{2023arXiv230113021F}.
Therefore, Theorem \ref{thm:existence_discrete} guarantees the existence of a unique discrete optimal solution $(\bar{\mathsf{y}}_{h},\bar{\mathsf{u}}_{h})\in Y_{h}\times U_{ad,h}$ of \eqref{def:discrete_opt_cont_prob_i}--\eqref{def:discrete_opt_cont_prob_ii}. Here, $\bar{\mathsf{y}}_{h}=(\bar{\mathsf{y}}_{o,h},\bar{\boldsymbol\sigma}_{h},\overline{\hat{\mathsf{y}}}_{o,h},\overline{\hat{\sigma}}_{h})$. Moreover, the discrete optimal control $\bar{\mathsf{u}}_{h}$ satisfies the discrete variational inequality $(\bar{\mathsf{p}}_{o,h} + \alpha\bar{\mathsf{u}}_{h}, \mathsf{u}_{h} - \bar{\mathsf{u}}_{h})_{L^2(\Omega)} \geq 0$ for all $\mathsf{u}_{h}\in U_{ad,h}$, where
$\bar{\lambda}_{h} := (\bar{\zeta}_{h},\bar{\boldsymbol\mu}_{h},\overline{\hat{\zeta}}_{h},\overline{\hat{\mu}}_{h})\in Y_{h}$ and $\bar{\mathsf{p}}_{h} :=(\bar{\mathsf{p}}_{o,h},\bar{\mathbf{p}}_{h}) \in V_{h}$ solve the discrete adjoint equation \eqref{eq:discrete_adj_eq} with $\mathsf{y}_{h}=\bar{\mathsf{y}}_{h}$.


\subsubsection{A priori error estimates}\label{sec:a_priori_Poisson}
In this section we assume that $\Omega\subset\mathbb{R}^{d}$ ($d=2,3)$ is a convex polygonal/polyhedral domain and that $\mathsf{y}_{d} \in H^1(\mathscr{T})$. In what follows we study regularity properties of the optimal solution to obtain an order of convergence for the error in terms of the discretization parameter $h$, associated to $\mathscr{T}$.

Let us start by estimating the term  $\inf_{w_{h}\in Y_{h}}\|\bar{\mathsf{y}} - w_{h}\|_{Y}^2$ in \eqref{eq:sharp_estimate_control}; cf. estimate \eqref{eq:estimate_control_apriori}. Since $\bar{\mathsf{y}}$ solves the Poisson equation with $\bar{\mathsf{u}}\in L^2(\Omega)$ as source term, we can invoke regularity results from \cite[Theorem 3.2.1.2]{MR775683} when $d=2$ and \cite[Section 4.3.1]{MR2641539} when $d=3$, to conclude that $\bar{\mathsf{y}}_{o}\in H_0^1(\Omega)\cap H^2(\Omega)$ and $\bar{\boldsymbol\sigma}\in H^1(\Omega)$. Hence, a direct application of \cite[Theorem 6]{MR3977484} yields that $\inf_{w_{h}\in Y_{h}}\|\bar{\mathsf{y}} - w_{h}\|_{Y} \lesssim h$.

We now estimate $\inf_{v_{h}\in V_{h}}\|\bar{\mathsf{p}} - v_{h}\|_{V}^2$ by proving  regularity properties for $\bar{\mathsf{p}} = (\bar{\mathsf{p}}_{o},\bar{\mathbf{p}}) \in V$. Since $\bar{\mathsf{p}}$ solves the Poisson equation with $\bar{\mathsf{y}}_{o} - \mathsf{y}_{d} \in L^2(\Omega)$ as source term, the convexity of $\Omega$ implies that $\bar{\mathsf{p}}_{o}\in H_0^1(\Omega)\cap H^2(\Omega)$ and $\bar{\mathbf{p}}\in H^1(\Omega)$; see \cite[Theorem 3.2.1.2]{MR775683} and \cite[Section 4.3.1]{MR2641539}. In addition, we have that div $\bar{\mathbf{p}}=\bar{\mathsf{y}}-\mathsf{y}_{d}\in H^{1}(\mathscr{T})$, since $\mathsf{y}_{d}\in H^1(\mathscr{T})$. Consequently, standard estimates for projection operators (see, e.g., \cite[Section 3.1 and eq. (3.1)]{MR3977484}) allow us to conclude that $\inf_{v_{h}\in V_{h}}\|\bar{\mathsf{p}} - v_{h}\|_{V}\lesssim h$.

Let us bound now  $\inf_{w_{h}\in Y_{h}}\|\bar{\lambda} - w_{h}\|_{Y}^2$ in \eqref{eq:sharp_estimate_control}. To study the regularity of $\bar{\lambda} = (\bar{\zeta},\bar{\boldsymbol\mu},\overline{\hat{\zeta}},\overline{\hat{\mu}})$, we first note that  \cite[Proposition 3.3]{MR4090394}
\begin{equation*}
\bar{\zeta} =  \bar{\mathsf{y}}_{o} - \mathsf{y}_{d} + \mathsf{e}, 
\qquad \bar{\boldsymbol\mu} = \bar{\mathbf{p}} + \mathbf{e}, 
\qquad \bar{\hat{\zeta}} = \mathsf{e}, 
\qquad \bar{\hat{\mu}} =   2\bar{\mathbf{p}}\cdot\mathbf{n} + \mathbf{e}\cdot\mathbf{n},
\end{equation*}  
where $\mathsf{e}\in H_0^1(\Omega)$ satisfies $-\Delta \mathsf{e} = \bar{\mathsf{p}}_o + 2(\bar{\mathsf{y}}_{o} - \mathsf{y}_{d})$ and $\mathbf{e} = -\nabla \mathsf{e}$. Moreover, since $\Omega$ is convex, we have that $\mathsf{e}\in H_0^1(\Omega)\cap H^2(\Omega)$ and that $\bar{\boldsymbol\mu}\in H^1(\Omega)^{d}$. Additionally, the fact that $\mathsf{y}_{d}\in H^1(\mathscr{T})$ implies that $\bar{\zeta}\in H^1(\mathscr{T})$. These regularities and the fact that $\bar{\mathsf{p}}_{o},\mathsf{e}\in H^2(\Omega)$ would seem enough to provide, at least, a convergence rate of order $h$ for $\inf_{w_{h}\in Y_{h}}\|\bar{\lambda} - w_{h}\|_{Y}$. Unfortunately, this is not what the authors show in \cite[Corollary 3.4]{MR4090394}: they require that $\bar{\mathsf{p}}_{o},\mathsf{e}\in H^3(\Omega)$ to obtain the previously mentioned rate of convergence $h$ (see \cite[estimate (63)]{MR4090394}). This in mainly due to the estimation of the error approximation associated to the term $\bar{\hat{\mu}}\in H^{-1/2}(\partial\mathscr{T})$. However, using the previous regularities and the arguments provided in \cite[Theorem 6]{MR3977484} to bound the norm $\|\cdot\|_{-1/2,\partial\mathscr{T}}$, we obtain the estimate  $\inf_{w_{h}\in Y_{h}}\|\bar{\lambda} - w_{h}\|_{Y}\lesssim h$.

To bound the terms $\|(1-\pi_{U_{h}})\bar{\mathsf{p}}_{o,h}\|_{L^2(\Omega)}^{2}$ and  $\|(1-\pi_{U_{h}})\bar{\mathsf{u}}\|_{L^2(\Omega)}^2$ in \eqref{eq:sharp_estimate_control}, we use that $\bar{\mathsf{p}}_{o,h},\bar{\mathsf{u}} \in H^{1}(\mathscr{T})$, and approximation properties of the projection operator $\pi_{U_{h}}$. These result in $\|(1-\pi_{U_{h}})\bar{\mathsf{p}}_{o,h}\|_{L^2(\Omega)} + \|(1-\pi_{U_{h}})\bar{\mathsf{u}}\|_{L^2(\Omega)}\lesssim h(\|\bar{\mathsf{p}}_{o,h}\|_{H^1(\mathscr{T})}+\|\nabla \bar{\mathsf{u}}\|_{L^2(\Omega)})$. We immediately note that the term $\|\bar{\mathsf{p}}_{o,h}\|_{H^1(\mathscr{T})}$ is uniformly bounded since $\|\bar{\mathsf{p}}_{o,h}\|_{H^1(\mathscr{T})} \leq \|\bar{\mathsf{p}}_{h}\|_{V} \lesssim \|\bar{\mathsf{y}}_{h}\|_{Y} + \|y_{d}\|_{L^2(\Omega)}\lesssim \max\{|\texttt{a}|,|\texttt{b}|\}|\Omega| + \|y_{d}\|_{L^2(\Omega)}$, where $|\Omega|$ denotes the Lebesgue measure of $\Omega$.

In view of the previous error estimates we have proved that $\|\bar{\mathsf{u}} - \bar{\mathsf{u}}_{h}\|_{L^2(\Omega)} \lesssim h$  which, in light of Theorem \ref{thm:a_priori_estimates}, yields $ \|\bar{\mathsf{u}} - \bar{\mathsf{u}}_{h}\|_{L^2(\Omega)} + \|\bar{\mathsf{y}} - \bar{\mathsf{y}}_{h}\|_{Y} + \|\bar{\mathsf{p}} - \bar{\mathsf{p}}_{h}\|_{V} \lesssim h$.

Finally, we note that the assumption $\mathsf{y}_{d}\in H^1(\mathscr{T})$ is quite restrictive and unusual when providing a priori error estimates for an optimal control problem subject to the Poisson equation. This fact motivates, under the framework of the DPG method, the design of a posteriori error estimators for this particular problem.


\subsubsection{A posteriori error estimates}\label{sec:a_posteriori_Poisson}

Let us introduce, on the basis of the optimal discrete adjoint state $\bar{\mathsf{p}}_{h} = (\bar{\mathsf{p}}_{o,h},\bar{\mathbf{p}}_{h})\in V_{h}$, $\tilde{\mathsf{u}}:= \Pi_{[\texttt{a},\texttt{b}]}(-\alpha^{-1}\bar{\mathsf{p}}_{o,h})\in U_{ad}$. It follows, from \cite[Lemma 2.26]{MR2583281}, that $\tilde{\mathsf{u}}$ solves the variational inequality \eqref{eq:ineq_tilde_u}. Consequently, the results of Theorems \ref{thm:reliability_estimates} and \ref{thm:efficiency_estimates} hold. 

We now give computable estimators that are equivalent to the residual terms present in Theorems \ref{thm:reliability_estimates} and \ref{thm:efficiency_estimates} (see Remark \ref{remark:error_estimators}). For the discretization of the control variable, we define the estimator $\eta_{ct}:=\|\tilde{\mathsf{u}} - \bar{\mathsf{u}}_{h}\|_{L^2(\Omega)}$. For the estimator associated to the discrete state equation \eqref{eq:equivalent_discrete_st_eq} we state the following lemma, whose proof follows along similar arguments to the ones developed in \cite[Theorem 2.1]{MR3215064}. We note that for the following result to hold, property \eqref{eq:Fortin_op_control} for the Fortin operator $\Pi$ is essential.

\begin{lemma}[equivalent representation]\label{lemma:equiv_rep}
We have that $\| D\bar{\mathsf{u}}_{h} - B\bar{\mathsf{y}}_{h}\|_{V'} \eqsim \| D\bar{\mathsf{u}}_{h} - B\bar{\mathsf{y}}_{h}\|_{V_{h}^{\prime}}$. Moreover, the latter term is equal to $\|\bar{\varepsilon}_{h}\|_{V}$, where $(\bar{\varepsilon}_{h},\bar{\mathsf{y}}_{h})\in V_{h}\times Y_{h}$ solves the discrete state equation \eqref{eq:equivalent_discrete_st_eq} with $\mathsf{u}_{h}$ replaced by $\bar{\mathsf{u}}_{h}$.
\end{lemma}
\begin{proof}
The last claim follows directly from~\eqref{eq:equivalent_discrete_st_eq}. It only remains to show that $\| D\bar{\mathsf{u}}_{h} - B\bar{\mathsf{y}}_{h}\|_{V'} \eqsim \| D\bar{\mathsf{u}}_{h} - B\bar{\mathsf{y}}_{h}\|_{V_{h}^{\prime}}$.
  We note that the following lines of proof only require the existence of a Fortin operator $\Pi$ (Definition~\ref{def:fortin_op}) that additionally satisfies \eqref{eq:Fortin_op_control}. Observe that
  \begin{align*}
    \| D\bar{\mathsf{u}}_{h} - B\bar{\mathsf{y}}_{h}\|_{V_h'} \leq \| D\bar{\mathsf{u}}_{h} - B\bar{\mathsf{y}}_{h}\|_{V'}
  \end{align*}
  by definition of the dual norm and $V_h \subseteq V$. The other direction can be seen from
  \begin{align*}
    \| D\bar{\mathsf{u}}_{h} - B\bar{\mathsf{y}}_{h}\|_{V'} 
    &= \sup_{0\neq v\in V} \frac{\langle D\bar{\mathsf{u}}_{h} - B\bar{\mathsf{y}}_{h},v\rangle_{V', V}}{\|v\|_V} 
    = \sup_{0\neq v\in V} \frac{\langle D\bar{\mathsf{u}}_{h} - B\bar{\mathsf{y}}_{h},\Pi v\rangle_{V', V}}{\|v\|_V} \\
    &\lesssim \sup_{0\neq v\in V} \frac{\langle D\bar{\mathsf{u}}_{h} - B\bar{\mathsf{y}}_{h},\Pi v\rangle_{V', V}}{\|\Pi v\|_V}
    \\ 
    &\leq \sup_{0\neq v\in V_h} \frac{\langle D\bar{\mathsf{u}}_{h} - B\bar{\mathsf{y}}_{h},v\rangle_{V', V}}{\|v\|_V} = \| D\bar{\mathsf{u}}_{h} - B\bar{\mathsf{y}}_{h}\|_{V_h'},
  \end{align*}
  where we have used the properties of $\Pi$. This concludes the proof.
\end{proof}

In light of the previous lemma, we define $\eta_{st}:=\|\bar{\varepsilon}_{h}\|_{V}$ as the error estimator associated to the discrete state equation.

Associated to the discrete adjoint equation, we define the error estimator
\begin{multline}\label{def:estimator_adj}
\eta_{adj}^2:=\|\text{div }\bar{\mathbf{p}}_{h} - \bar{\mathsf{y}}_{o,h} + \mathsf{y}_{d}\|_{L^2(\mathscr{T})}^2 + \|\bar{\mathbf{p}}_{h} + \nabla\bar{\mathsf{p}}_{o,h}\|_{L^{2}(\mathscr{T})^{d}}^{2} \\ +\sum_{\gamma\in \mathcal{E}_{\text{int}}}h_{\gamma}\|\llbracket \bar{\mathbf{p}}_{h}\cdot \mathbf{n}\rrbracket\|_{L^2(\gamma)}^2 + \sum_{\gamma\in \mathcal{E}}h_{\gamma}^{-1}\|\llbracket \bar{\mathsf{p}}_{o,h}\mathbf{n}\rrbracket\|_{L^2(\gamma)}^2.
\end{multline}
The use of \cite[Theorem 4.1]{MR4090394} yields that $\|O^{*}\mathcal{I}_{\mathcal{Y}}^{\mathcal{Y}'}(\bar{\mathsf{y}}_{o,h} - \mathsf{y}_{d}) - B^{*}\bar{\mathsf{p}}_{h}\|_{Y'}^{2} = \|\text{div }\bar{\mathbf{p}}_{h} - \bar{\mathsf{y}}_{o,h} + \mathsf{y}_{d}\|_{L^2(\Omega)}^{2} + \|\bar{\mathbf{p}}_{h} + \nabla\bar{\mathsf{p}}_{o,h}\|_{L^{2}(\Omega)^{d}}^{2} + \eta(\mathsf{p}_{h})^{2}$, with $\eta(\mathsf{p}_{h})$ defined as in \cite[Eq. (69)]{MR4090394}. Moreover, the arguments provided in the proof of \cite[Theorem 4.2]{MR4090394} reveal  
\begin{equation*}
\eta(\mathsf{p}_{h})^{2} \eqsim \sum_{\gamma\in \mathcal{E}_{\text{int}}}h_{\gamma}\|\llbracket \bar{\mathbf{p}}_{h}\cdot \mathbf{n}\rrbracket\|_{L^2(\gamma)}^2 + \sum_{\gamma\in \mathcal{E}}h_{\gamma}^{-1}\|\llbracket \bar{\mathsf{p}}_{o,h}\mathbf{n}\rrbracket\|_{L^2(\gamma)}^2.
\end{equation*}
We have thus concluded that $\eta_{adj} \eqsim \|O^{*}\mathcal{I}_{\mathcal{Y}}^{\mathcal{Y}'}(\bar{\mathsf{y}}_{o,h} - \mathsf{y}_{d}) - B^{*}\bar{\mathsf{p}}_{h}\|_{Y'}$.

Finally, we notice that the Lipschitz property of $\Pi_{[\texttt{a},\texttt{b}]}$ implies that $\|\bar{\mathsf{u}} - \tilde{\mathsf{u}}\|_{L^2(\Omega)} \leq \alpha^{-1}\|\bar{\mathsf{p}}_{o} - \bar{\mathsf{p}}_{o,h}\|_{L^2(\Omega)} \lesssim \| \bar{\mathsf{p}} - \bar{\mathsf{p}}_{h}\|_{V}$.

Summarizing the above observations, we have proved, in view of Theorems \ref{thm:reliability_estimates} and \ref{thm:efficiency_estimates}, the following result.

\begin{corollary}[a posteriori error estimation]\label{coro:a_posteriori_estimation}
Let $\eta^2:= \eta_{ct}^2 + \eta_{st}^2 + \eta_{adj}^2$. Then, we have that $\eta \eqsim  \|\bar{\mathsf{u}} - \bar{\mathsf{u}}_{h}\|_{L^2(\Omega)} +  \|\bar{\mathsf{y}} - \bar{\mathsf{y}}_{h}\|_{Y} + \|\bar{\mathsf{p}} - \bar{\mathsf{p}}_{h}\|_{V}$.
\end{corollary}


\subsection{Stokes problem}\label{sec:Stokes_problem}
We study a control problem subject to physical equations for incompressible Newtonian flow. To present the state equation we introduce, for each $\mathsf{v}\in L^2(\Omega)^{d}$ and $\mathbf{v}\in L^2(\Omega)^{d\times d}$, the linear Green strain and the trace-free deviatoric part of $\mathbf{v}$ as $\epsilon(\mathsf{v}):=2^{-1}(\nabla \mathsf{v} + \nabla \mathsf{v}^{\intercal})$ and dev$(\mathbf{v}):=\mathbf{v}-d^{-1}\text{Tr}(\mathbf{v})\mathbf{I}$, respectively. Here, $\mathbf{I}$ denotes the identity matrix of order $d$. We consider the \emph{stress-velocity} formulation for our state equation:
\begin{equation}\label{eq:stress-velocity}
-\text{div } \boldsymbol\sigma = \mathsf{u} \text{ in } \Omega, \quad \text{dev}(\boldsymbol\sigma) - \epsilon(\mathsf{y}_{o}) = 0 \text{ in } \Omega, \quad \mathsf{y}_{o} = \mathbf{0} \text{ on } \partial \Omega, \quad \int_{\Omega}\text{Tr}(\boldsymbol\sigma) = 0.
\end{equation}

To present the spaces that we will use for the optimal control problem subject to \eqref{eq:stress-velocity}, we introduce $L_{\text{sym}}^{2}(\Omega)^{d\times d}:=\{ \mathbf{v}\in L^{2}(\Omega)^{d\times d}: \mathbf{v}= \mathbf{v}^{\intercal}\}$ and $H_{\text{sym}}(\text{div},\mathscr{T}):=\{ \mathbf{v}\in H(\text{div},\mathscr{T})^{d}: \mathbf{v}= \mathbf{v}^{\intercal}\}$. Hence, the spaces considered for \eqref{def:opt_cont_prob} are:
\[
U = \mathcal{Y} = L^2(\Omega)^{d}, 
\quad Y = L^2(\Omega)^{d} \times L^2_{\text{sym}}(\Omega)^{d\times d} \times H_{0}^{1/2}(\partial\mathscr{T})^{d} \times H^{-1/2}(\partial\mathscr{T})^{d}\times \mathbb{R}, 
\]
and $V = H^1(\mathscr{T})^{d} \times H_{\text{sym}}(\text{div},\mathscr{T})\times \mathbb{R}$. We equip the previous spaces with their corresponding natural norms. On the other hand, the admissible set $U_{ad}$ is defined by $U_{ad}:=\{u \in L^2(\Omega)^{d}:  \texttt{a}_{i} \leq u_{i}(x) \leq \texttt{b}_{i} \text{ a.e. in } \Omega \text{ for all } i=1,\ldots,d.\}$, where the vectors $\texttt{a},\texttt{b} \in \mathbb{R}^{d}$ satisfy $\texttt{a}_{i} < \texttt{b}_{i}$ for all $i\in \{1,\ldots,d\}$. 

We define the bilinear form $b:Y\times V\to \mathbb{R}$ as follows \cite[Eq. (3.12)]{MR3215064}:
\begin{multline}\label{def:bilinear_Stokes}
b(\mathsf{y},v) := (\text{dev}(\boldsymbol\sigma), \mathbf{v})_{\mathscr{T}} + (\mathsf{y}_{o}, \text{div }\mathbf{v})_{\mathscr{T}} - \langle \hat{\mathsf{y}}_{o}, \mathbf{v}\cdot \mathbf{n} \rangle_{\partial\mathscr{T}} + (\alpha,\text{Tr}(\mathbf{v}))_{\mathscr{T}} \\ +(\boldsymbol\sigma, \epsilon(\mathsf{v}))_{\mathscr{T}} - \langle \hat \sigma, \mathsf{v}\rangle_{\partial\mathscr{T}} + (\text{Tr}(\boldsymbol\sigma),\beta)_{\mathscr{T}},
\end{multline}
where $\mathsf{y}=(\mathsf{y}_{o},\boldsymbol\sigma,\hat{\mathsf{y}}_{o},\hat{\sigma},\alpha)\in Y$, $v=(\mathsf{v},\mathbf{v},\beta)\in V$, $(\cdot,\cdot)_{\mathscr{T}}:=\sum_{T\in\mathscr{T}}(\cdot,\cdot)_{L^2(T)^{d}}$, and $\langle \cdot,\cdot\rangle_{\partial\mathscr{T}}:=\sum_{T\in\mathscr{T}}\langle \cdot,\cdot\rangle_{\partial T}$.

The observation operator $O$ is defined by $O\mathsf{y}=\mathsf{y}_{o}$ and the mapping $D$ is defined as the continuous embedding $L^2(\Omega)^d\hookrightarrow L^2(\Omega)^d\times \{0\} \hookrightarrow  V'$. Consequently, the state equation reads: $b(\mathsf{y},v) = (\mathsf{u},\mathsf{v})_{L^2(\Omega)}$ for all $v=(\mathsf{v},\mathbf{v})\in V$.

Theorem 3.7 from \cite{MR3215064} guarantees that the bilinear form \eqref{def:bilinear_Stokes} satisfies the inf-sup condition \eqref{eq:inf_sup_b}. Hence, Theorems \ref{thm:existence} and \ref{thm:1st_order_opt_cond} prove the existence and uniqueness of an optimal solution pair $(\bar{\mathsf{y}},\bar{\mathsf{u}})\in Y\times U_{ad}$ for problem \eqref{def:opt_cont_prob} satisfying the first order optimality condition \eqref{eq:var_ineq}. We can rewrite such an optimality condition as $(\bar{\mathsf{p}}_{o} + \alpha\bar{\mathsf{u}}, \mathsf{u} - \bar{\mathsf{u}})_{L^2(\Omega)^{d}} \geq 0$ for all $\mathsf{u}\in U_{ad}$, where $\bar{\mathsf{p}} :=(\bar{\mathsf{p}}_{o},\bar{\mathbf{p}},\bar{\gamma}) \in V$ and $\bar{\lambda} := (\bar{\zeta},\bar{\boldsymbol\mu},\overline{\hat{\zeta}},\overline{\hat{\mu}},\bar{\delta})\in Y$ solve the adjoint equation \eqref{eq:adj_eq_mixed}. Moreover, we have a characterization for $\bar{\mathsf{u}}$:
\begin{equation*}
\bar{\mathsf{u}}_{i}(x) = \Pi_{[\texttt{a},\texttt{b}]}(-\alpha^{-1}\bar{\mathsf{p}}_{o,i}(x)) \quad \text{for a.e. } x\in\Omega, ~ i \in \{1, \ldots,d \}.
\end{equation*}
Notice that $\bar{\mathsf{u}} \in H^{1}(\mathscr{T})^{d}$  \cite[Theorem A.1]{MR1786735}.


\subsubsection{Discrete approximation}

The next conforming discrete spaces are considered:
\begin{equation*}
Y_{h}:=\mathcal{P}^{k}(\mathscr{T})^{d}\times \mathcal{P}^{k}(\mathscr{T})^{d\times d}\times \mathcal{P}^{k+1}_{c,0}(\partial\mathscr{T})^{d}\times \mathcal{P}^{k}(\partial\mathscr{T})^{d}, \quad V_{h}:=\mathcal{P}^{k_{1}}(\mathscr{T})^{d}\times\mathcal{P}^{k_{2}}(\mathscr{T})^{d \times d},
\end{equation*}
$\mathcal{Y}_{h}:=\mathcal{P}^{k}(\mathscr{T})^{d}$, and $U_{h}:=\mathcal{P}^{k}(\mathscr{T})^{d}$, with $k\in\mathbb{N}_{0}$ and $k_{1},k_{2}\in\mathbb{N}$ satisfying $k_{1} \geq k + d + 1$ and $k_{2} \geq k + 2$.  A Fortin operator for the previous choice of spaces exists. Writing $\Pi v = (\Pi^\nabla \mathsf{v},\Pi^{\mathrm{div},\mathrm{sym}}\mathbf{v})$ for $v = (\mathsf{v},\mathbf{v})\in V$. Here, $\Pi^\nabla$ denotes the operator from~\cite[Section~3.1]{2023arXiv230113021F} applied componentwise and $\Pi^{\mathrm{div},\mathrm{sym}}$ is the operator from~\cite[Lemma~4.1]{MR3143683}. We stress that $\Pi$ satisfies the statements from Defintion~\ref{def:fortin_op} and identity~\eqref{eq:Fortin_op_control}. We thus have, in view of Theorem \ref{thm:existence_discrete}, that there exists a unique discrete optimal solution $(\bar{\mathsf{y}}_{h},\bar{\mathsf{u}}_{h})\in Y_{h}\times U_{ad,h}$ for problem \eqref{def:discrete_opt_cont_prob_i}--\eqref{def:discrete_opt_cont_prob_ii}. Here, $\bar{\mathsf{y}}_{h}=(\bar{\mathbf{y}}_{h},\bar{\boldsymbol\sigma}_{h},\overline{\hat{\mathbf{y}}}_{h},\overline{\hat{\sigma}}_{h},\bar{\alpha}_{h})$. Additionally, $\bar{\mathsf{u}}_{h}$ satisfies the discrete variational inequality $(\bar{\mathsf{p}}_{o,h} + \alpha\bar{\mathsf{u}}_{h}, \mathsf{u}_{h} - \bar{\mathsf{u}}_{h})_{L^2(\Omega)^{d}} \geq 0$ for all $\mathsf{u}_{h}\in U_{ad,h}$, with $\bar{\mathsf{p}}_{h} :=(\bar{\mathsf{p}}_{o,h},\bar{\mathbf{p}}_{h},\bar{\gamma}_{h}) \in V_{h}$ and 
$\bar{\lambda}_{h} := (\bar{\zeta}_{h},\bar{\boldsymbol\mu}_{h},\overline{\hat{\zeta}}_{h},\overline{\hat{\mu}}_{h},\bar{\delta}_{h})\in Y_{h}$.


\subsubsection{A priori error estimates}

Assume that $\Omega\subset\mathbb{R}^{d}$ ($d=2,3)$ is a convex polygonal/polyhedral domain and that the desired state $\mathsf{y}_{d}\in H^1(\mathscr{T})^{d}$. We proceed as in the Poisson example and study regularity properties of the optimal solution.

We start with the term $\inf_{w_{h}\in Y_{h}}\|\bar{\mathsf{y}} - w_{h}\|_{Y}^2$ in \eqref{eq:sharp_estimate_control}; cf. estimate \eqref{eq:estimate_control_apriori}. Since $\bar{\mathsf{y}}$ solves the Stokes problem with a source term in $L^2(\Omega)$, we use \cite[Theorem 2]{MR0404849} and \cite{MR775683} when $d = 2$, and \cite{MR977489} and \cite{MR1301452} when $d=3$ to obtain that $\bar{\mathsf{y}}_{o}\in H_0^1(\Omega)^{d}\cap H^2(\Omega)^{d}$ and $\bar{\boldsymbol\sigma}\in H^1(\Omega)^{d\times d}$. The use of \cite[Theorem 6]{MR3977484} yields that $\inf_{w_{h}\in Y_{h}}\|\bar{\mathsf{y}} - w_{h}\|_{Y} \lesssim h$.

To bound the term $\inf_{v_{h}\in V_{h}}\|\bar{\mathsf{p}} - v_{h}\|_{V}^2$, we first note that $\bar{\mathsf{p}} = (\bar{\mathsf{p}}_{o},\bar{\mathbf{p}},\bar{\gamma}) \in V$ solves the Stokes equation with $\bar{\mathsf{y}}_{o} - \mathsf{y}_{d} \in L^2(\Omega)^{d}$ as source term. As a consequence of the convexity of $\Omega$, we conclude that $\bar{\mathsf{p}}_{o}\in H_0^1(\Omega)^{d} \cap H^2(\Omega)^{d}$ and $\bar{\mathbf{p}}\in H^1(\Omega)^{d\times d}$. Thus, using that $\mathsf{y}_{d}\in H^1(\mathscr{T})^{d}$, it follows that  div $\bar{\mathbf{p}}=\bar{\mathsf{y}}-\mathsf{y}_{d}\in H^{1}(\mathscr{T})^{d}$. Consequently, standard estimates for projection operators \cite[Section 3.1 and eq. (3.1)]{MR3977484} allow us to prove that $\inf_{v_{h}\in V_{h}}\|\bar{\mathsf{p}} - v_{h}\|_{V}\lesssim h$. 

Let us estimate $\inf_{w_{h}\in Y_{h}}\|\bar{\lambda} - w_{h}\|_{Y}^2$ in \eqref{eq:sharp_estimate_control}. To prove regularity properties for $\bar{\lambda} = (\bar{\zeta},\bar{\boldsymbol\mu},\overline{\hat{\zeta}},\overline{\hat{\mu}},\bar{\delta})$, we use the arguments elaborated in \cite[Proposition 3.3]{MR4090394} to obtain 
\begin{equation*}
\bar{\zeta} =  \bar{\mathsf{y}}_{o} - \mathsf{y}_{d} + \mathsf{e}, 
\quad \bar{\boldsymbol\mu} = \bar{\mathbf{p}} + \mathbf{e}, 
\quad \bar{\hat{\zeta}} = \mathsf{e}, 
\quad \bar{\hat{\mu}} =   \bar{\mathbf{p}}\cdot\mathbf{n} -  \nabla\bar{\mathsf{p}}_{o}\cdot\mathbf{n} + \mathbf{e}\cdot\mathbf{n},
\end{equation*}  
and $\bar{\delta} = d^{-1}\int_{\Omega}\text{Tr}(\bar{\mathbf{p}})+\int_{\Omega}\text{Tr}(\bar{\mathbf{e}})$. Here, $\mathsf{e}\in H_0^1(\Omega)^{d}$ satisfies $-\text{div } \mathbf{e} = \bar{\mathsf{p}}_o - \Delta \bar{\mathsf{p}}_{o} + \bar{\mathsf{y}}_{o} - \mathsf{y}_{d}$ and $\text{dev}(\mathbf{e}) - \epsilon(\mathsf{e}) = 0$. Moreover, since $\Omega$ is convex, we have that $\mathsf{e}\in H_0^1(\Omega)^{d}\cap H^2(\Omega)^{d}$ and that $\bar{\boldsymbol\mu}\in H^1(\Omega)^{d\times d}$. Additionally, the fact that $\mathsf{y}_{d}\in H^1(\mathscr{T})$ implies that $\bar{\zeta}\in H^1(\mathscr{T})$. Therefore, using the previous regularities and the arguments provided in \cite[Theorem 6]{MR3977484} to bound the norm $\|\cdot\|_{-1/2,\partial\mathscr{T}}$, we conclude that $\inf_{w_{h}\in Y_{h}}\|\bar{\lambda} - w_{h}\|_{Y}\lesssim h$.

We estimate the terms $\|(1-\pi_{U_{h}})\bar{\mathsf{p}}_{o,h}\|_{L^2(\Omega)^{d}}^{2}$ and $\|(1-\pi_{U_{h}})\bar{\mathsf{u}}\|_{L^2(\Omega)^{d}}^2$ in \eqref{eq:sharp_estimate_control} by using that $\bar{\mathsf{p}}_{o,h},\bar{\mathsf{u}} \in H^{1}(\mathscr{T})^{d}$ and approximation properties of the projection operator.

Hence, by virtue of Remark \ref{remark:estimation_eu} and Theorem \ref{thm:a_priori_estimates}, we conclude the a priori error bound $\|\bar{\mathsf{u}} - \bar{\mathsf{u}}_{h}\|_{L^2(\Omega)^{d}} + \|\bar{\mathsf{y}} - \bar{\mathsf{y}}_{h}\|_{Y} + \|\bar{\mathsf{p}} - \bar{\mathsf{p}}_{h}\|_{V} \lesssim h$.


\subsubsection{A posteriori error estimates}\label{sec:a_posteriori_Stokes}
For the discretization of the control variable, we define the error estimator $\eta_{ct}:=\|\tilde{\mathsf{u}} - \bar{\mathsf{u}}_{h}\|_{L^2(\Omega)^{d}}$. where $\tilde{\mathsf{u}}:= \Pi_{[\texttt{a},\texttt{b}]}(-\alpha^{-1}\bar{\mathsf{p}}_{o,h})$ is such that it satisfies the variational inequality \eqref{eq:ineq_tilde_u} \cite[Lemma 2.26]{MR2583281}. We note that the results of Theorems \ref{thm:reliability_estimates} and \ref{thm:efficiency_estimates} hold. 

We use $\eta_{st}=\|\bar{\varepsilon}_{h}\|_{V}$ as error estimator associated to the discrete state equation \eqref{eq:equivalent_discrete_st_eq}. We recall that $(\bar{\mathsf{y}}_{h},\bar{\varepsilon}_{h})\in Y_{h}\times V_{h}$ solves the discrete state equation \eqref{eq:equivalent_discrete_st_eq} with $\mathsf{u}_{h}$ replaced by $\bar{\mathsf{u}}_{h}$. We note that, since the Fortin operator $\Pi$ satisfies \eqref{eq:Fortin_op_control}, then Lemma \ref{lemma:equiv_rep} holds within the considered setting. Consequently, $\eta_{st} \eqsim \|D\bar{\mathsf{u}}_{h} - B\bar{\mathsf{y}}_{h}\|_{V'}$.

Associated to the discrete adjoint equation, we define the error estimator
\begin{multline*}
\eta_{adj}^2:=\|\text{div }\bar{\mathbf{p}}_{h} - \bar{\mathsf{y}}_{o,h} + \mathsf{y}_{d}\|_{L^2(\mathscr{T})^{d}}^2 + \|\text{dev}(\bar{\mathbf{p}}_{h}) + \epsilon(\bar{\mathsf{p}}_{o,h})\|_{L^{2}(\mathscr{T})^{d\times d}}^{2} \\ + \|\text{div }\bar{\mathsf{p}}_{o,h}\|_{L^2(\mathscr{T})}^{2} +\sum_{\gamma\in \mathcal{E}_{\text{int}}}h_{\gamma}\|\llbracket \bar{\mathbf{p}}_{h}\cdot \mathbf{n}\rrbracket\|_{L^2(\gamma)^{d}}^2 + \sum_{\gamma\in \mathcal{E}}h_{\gamma}^{-1}\|\llbracket \bar{\mathsf{p}}_{o,h}\mathbf{n}\rrbracket\|_{L^2(\gamma)^{d}}^2.
\end{multline*}
An application of \cite[Theorems 4.1 and 4.2]{MR4090394} yields $\eta_{adj} \eqsim \|O^{*}\mathcal{I}_{\mathcal{Y}}^{\mathcal{Y}'}(\bar{\mathsf{y}}_{o,h} - \mathsf{y}_{d}) - B^{*}\bar{\mathsf{p}}_{h}\|_{Y'}$.

The Lipschitz property of $\Pi_{[\texttt{a},\texttt{b}]}$ implies that $\|\bar{\mathsf{u}} - \tilde{\mathsf{u}}\|_{L^2(\Omega)^{d}} \lesssim \| \bar{\mathsf{p}} - \bar{\mathsf{p}}_{h}\|_{V}$.

The previous estimates lead us to conclude that $\eta \eqsim  \|\bar{\mathsf{u}} - \bar{\mathsf{u}}_{h}\|_{L^2(\Omega)^{d}} +  \|\bar{\mathsf{y}} - \bar{\mathsf{y}}_{h}\|_{Y} + \|\bar{\mathsf{p}} - \bar{\mathsf{p}}_{h}\|_{V}$, where $\eta^2:= \eta_{ct}^2 + \eta_{st}^2 + \eta_{adj}^2$ (cf. Corollary \ref{coro:a_posteriori_estimation}).

\section{Numerical examples}\label{sec:num_examples} 
In this section we present numerical examples in 2D that support our theoretical findings for the two model problems. For all examples we consider lowest-order discretizations, i.e., we choose $k=0$, $k_1=3$ and $k_2=2$, and employ the active set strategy from~\cite{KKT03} to solve the discrete variational inequalities.
In order to simplify the construction of exact optimal solutions, we have incorporated an extra forcing term $f\in L^2(\Omega)$ in the state equation. The right-hand side of the state equation reads as follows: $(f  + \bar{\mathsf{u}}, \mathsf{v})_{L^2(\Omega)}$.
 
\subsection{Problem in convex domain subject to Poisson equation}\label{sec:num_examples:poisson}
We consider the problem setup of Section~\ref{sec:Poisson_problem} with manufactured solution
\begin{align*}
  \bar{\mathsf{y}}_{o}(x,y) = \sin(\pi x)\sin(\pi y), \quad \bar{\mathsf{p}}_{o}(x,y) = 16 x(1-x)y(1-y) \quad \text{for } (x,y)\in \Omega := (0,1)^2.
\end{align*}
Set $\texttt{b}= - \texttt{a} =10$, $\alpha=10^{-2}$, and compute for $(x,y)\in \Omega$: $\bar{\mathsf{u}}(x,y) = \Pi_{[\texttt{a},\texttt{b}]}(-\alpha^{-1}\bar{\mathsf{p}}_{o}(x))$,
\begin{equation*}
f(x,y)= -\Delta \bar{\mathsf{y}}_{o}(x,y) -\bar{\mathsf{u}}(x,y),  \quad   \mathsf{y}_d(x,y) = \Delta \bar{\mathsf{p}}_{o}(x,y) + \bar{\mathsf{y}}_{o}(x,y).
\end{equation*}

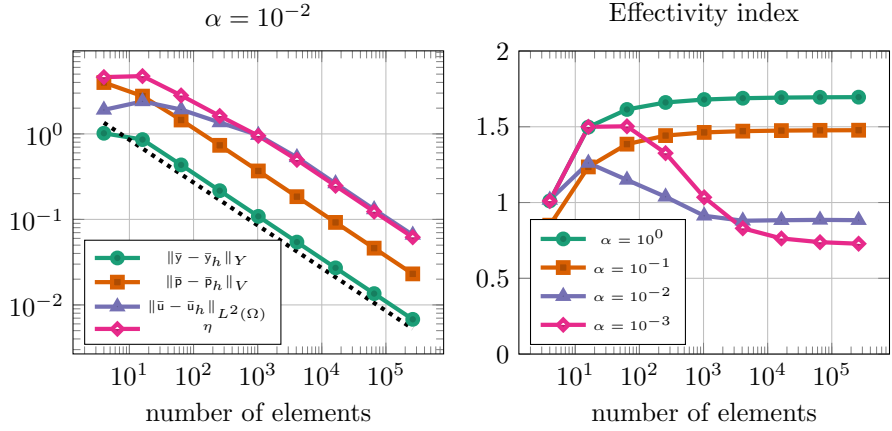
\begin{figure}
  \begin{tikzpicture}
  \begin{groupplot}[group style={group size= 2 by 1},width=0.5\textwidth,cycle list/Dark2-6,
                      cycle multiindex* list={
                          mark list*\nextlist
                          Dark2-6\nextlist},
                      every axis plot/.append style={ultra thick},
                      grid=major,
                      xlabel={number of elements},]
         \nextgroupplot[title={$\alpha=10^{-2}$},ymode=log,xmode=log,
           legend entries={\tiny{$\|\bar{\mathsf{y}}-\bar{\mathsf{y}}_h\|_Y$},\tiny{$\|\bar{\mathsf{p}}-\bar{\mathsf{p}}_h\|_V$},\tiny{$\|\bar{\mathsf{u}}-\bar{\mathsf{u}}_h\|_{L^2(\Omega)}$},\tiny $\eta$},
                      legend pos=south west]
                \addplot table [x=nE,y=errY] {data/ExamplePoissonSmooth_1e-02.dat};
                \addplot table [x=nE,y=errP] {data/ExamplePoissonSmooth_1e-02.dat};
                \addplot table [x=nE,y=errU] {data/ExamplePoissonSmooth_1e-02.dat};
                \addplot table [x=nE,y=estTotal] {data/ExamplePoissonSmooth_1e-02.dat};
                \addplot [black,dotted,mark=none] table [x=nE,y expr={2.7*sqrt(\thisrowno{0})^(-1)}] {data/ExamplePoissonSmooth_1e-02.dat};
       \nextgroupplot[title={Effectivity index},xmode=log,ymin=0,ymax=2,
         legend entries={\tiny $\alpha=10^0$,\tiny $\alpha=10^{-1}$,\tiny $\alpha=10^{-2}$,\tiny $\alpha=10^{-3}$},
                      legend pos=south west]
                \addplot table [x=nE,y=effInd] {data/ExamplePoissonSmooth_1e+00.dat};
                \addplot table [x=nE,y=effInd] {data/ExamplePoissonSmooth_1e-01.dat};
                \addplot table [x=nE,y=effInd] {data/ExamplePoissonSmooth_1e-02.dat};
                \addplot table [x=nE,y=effInd] {data/ExamplePoissonSmooth_1e-03.dat};
    \end{groupplot}
\end{tikzpicture}
  \caption{Errors and estimator for $\alpha = 10^{-2}$ (left) and effectivity indices (right) for the problem from Section~\ref{sec:num_examples:poisson}. The black dotted line indicates $\mathcal{O}( (\#\mathscr{T})^{-1/2})$.}\label{fig:PoissonSmooth}
\end{figure}

Figure~\ref{fig:PoissonSmooth} shows the error in the field variables in the state, the error of the adjoint state, and the $L^2(\Omega)$-error of the control compared to the total error estimator $\eta$. Given that the exact solutions are smooth our a priori error analysis predicts $\mathcal{O}(h)$ convergence which is also seen in our experiment. 
Furthermore, the right plot of Figure~\ref{fig:PoissonSmooth} shows the effectivity index 
\begin{equation*}
  \frac{\eta}{\sqrt{\|(\bar{\mathsf{y}}_o-\bar{\mathsf{y}}_{o,h},\bar{\boldsymbol\sigma} - \bar{\boldsymbol\sigma}_h)\|_{L^2(\Omega)}^2+\|\bar{\mathsf{p}} - \bar{\mathsf{p}}_h\|_V^2+\|\bar{\mathsf{u}} - \bar{\mathsf{u}}_h\|_{L^2(\Omega)}^2}}
\end{equation*}
for $\alpha =10^{-j}$, $j=0,1,2,3$. It can be observed that the effectivity index depends on $\alpha$ indicating that the error estimator is not uniformly equivalent to the total error. However, we note that the effectivity index ranges between $0.7$ and $1.7$.

\subsection{Problem in non-convex domain subject to Poisson equation}\label{sec:num_examples:Lshape}
We consider the problem setup of Section~\ref{sec:Poisson_problem} and data
\begin{align*}
\mathsf{y}_d(x,y) &= 1 \quad\text{for }(x,y) \in (-1,1)^2\setminus [-1,0]^2, 
  \quad \texttt{a} = 0.1, ~ \texttt{b} = 0.12,~ \alpha = 1.
\end{align*}
This example is taken from~\cite[Section~4.4]{LSQoptimal}.
Here, we do not know the exact solution in explicit form but note that reduced regularities are expected due the reentrant corner of the L-shaped domain.

\begin{figure}
\begin{center}\begin{tikzpicture}
\begin{loglogaxis}[
    width=0.49\textwidth,
cycle list/Dark2-6,
cycle multiindex* list={
mark list*\nextlist
Dark2-6\nextlist},
every axis plot/.append style={ultra thick},
xlabel={number of elements},
grid=major,
legend entries={\small $\eta$ unif.,\small $\eta$ adap.},
legend pos=south west,
]

\addplot table [x=nE,y=estTotal] {data/ExamplePoissonSingUnif.dat};
\addplot table [x=nE,y=estTotal] {data/ExamplePoissonSingAdap.dat};
\addplot [black,dotted,mark=none] table [x=nE,y expr={4*sqrt(\thisrowno{0})^(-1)}] {data/ExamplePoissonSingAdap.dat};
\end{loglogaxis}
\end{tikzpicture}

\begin{tikzpicture}
\begin{axis}[hide axis,
width=0.49\textwidth,
    axis equal,
]

\addplot[patch,color=white,
faceted color = black, line width = 0.2pt,
patch table ={data/ele5.dat}] file{data/coo5.dat};
\end{axis}
\end{tikzpicture}
\begin{tikzpicture}
\begin{axis}[hide axis,
width=0.49\textwidth,
    axis equal,
]

\addplot[patch,color=white,
faceted color = black, line width = 0.2pt,
patch table ={data/ele6.dat}] file{data/coo6.dat};
\end{axis}
\end{tikzpicture}
\begin{tikzpicture}
\begin{axis}[hide axis,
width=0.49\textwidth,
    axis equal,
]

\addplot[patch,color=white,
faceted color = black, line width = 0.2pt,
patch table ={data/ele7.dat}] file{data/coo7.dat};
\end{axis}
\end{tikzpicture}
\end{center}
\caption{Estimators for uniform and adaptive mesh-refinement as well as meshes generated by the adaptive algorithm for the problem from Section~\ref{sec:num_examples:Lshape}. The black dotted line indicates $\mathcal{O}( (\#\mathscr{T})^{-1/2})$.}\label{fig:PoissonSing}
\end{figure}
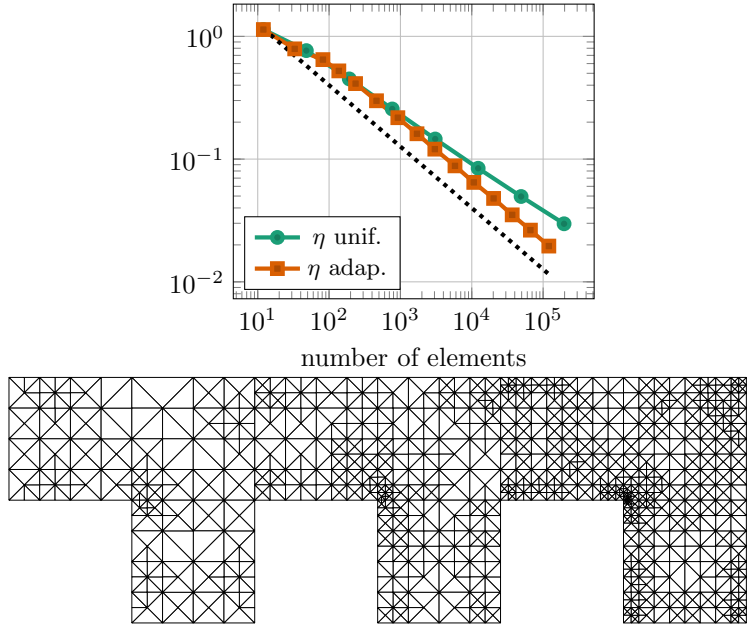

Figure~\ref{fig:PoissonSing} compares the estimator for uniform and adaptive mesh refinement. We employed a standard adaptive loop where we used the bulk criterion
\begin{align*}
  \frac12 \eta^2 \leq \sum_{T\in\mathcal{M}} \eta(K)^2
\end{align*}
to mark a (minimal) set $\mathcal{M}$ of elements for refinement.
It can be observed that uniform mesh refinements lead to reduced convergence rates whereas the adaptive algorithm recovers $\mathcal{O}( (\#\mathscr{T})^{-1/2})$ convergence.

\subsection{Problem in convex domain subject to Stokes equation}\label{sec:num_examples:stokes}
We consider the problem setup of Section~\ref{sec:Stokes_problem} with manufactured solution
\begin{alignat*}{2}
  \bar{\mathsf{y}}_{o}(x,y) &= \operatorname{curl}\big( x(1-x)y(1-y) \big)^2,  &\bar{\mathsf{p}}_{o}(x,y) \,=\,& \operatorname{curl}\big( \sin(\pi x)\sin(\pi y)\big)^2, \\
  \boldsymbol\sigma(x,y) &= \epsilon(\bar{\mathsf{y}}_{o})(x,y) - (x-\tfrac12)\begin{pmatrix} 1 & 0 \\ 0 & 1\end{pmatrix}, 
  &\bar{\mathbf{p}}(x,y) =\,& \epsilon(\bar{\mathsf{p}}_{o})(x,y)
\end{alignat*}
for $(x,y)\in \Omega := (0,1)^2$, set $\texttt{a}=-10$, $\texttt{b}=10$, and compute the control and other quantities thereof.

\begin{figure}
\begin{tikzpicture}
  \begin{groupplot}[group style={group size= 2 by 1},width=0.5\textwidth,cycle list/Dark2-6,
                      cycle multiindex* list={
                          mark list*\nextlist
                          Dark2-6\nextlist},
                      every axis plot/.append style={ultra thick},
                      grid=major,
                      xlabel={number of elements},]
                      \nextgroupplot[title={$\alpha=10^{-2}$},ymode=log,xmode=log,ymin=10^(-2),ymax=10^(3),
           legend entries={\tiny{$\|\bar{\mathsf{y}}-\bar{\mathsf{y}}_h\|_Y$},\tiny{$\|\bar{\mathsf{p}}-\bar{\mathsf{p}}_h\|_V$},\tiny{$\|\bar{\mathsf{u}}-\bar{\mathsf{u}}_h\|_{L^2(\Omega)^{d}}$},\tiny $\eta$},
                      legend pos=north east]
                \addplot table [x=nE,y=errY] {data/ExampleStokesSmooth_1e-02.dat};
                \addplot table [x=nE,y=errP] {data/ExampleStokesSmooth_1e-02.dat};
                \addplot table [x=nE,y=errU] {data/ExampleStokesSmooth_1e-02.dat};
                \addplot table [x=nE,y=estTotal] {data/ExampleStokesSmooth_1e-02.dat};
                \addplot [black,dotted,mark=none] table [x=nE,y expr={50*sqrt(\thisrowno{0})^(-1)}] {data/ExampleStokesSmooth_1e-02.dat};
       \nextgroupplot[title={Effectivity index},xmode=log,ymin=0,ymax=2,
         legend entries={\tiny $\alpha=10^0$,\tiny $\alpha=10^{-1}$,\tiny $\alpha=10^{-2}$,\tiny $\alpha=10^{-3}$},
                      legend pos=south west]
                \addplot table [x=nE,y=effInd] {data/ExampleStokesSmooth_1e+00.dat};
                \addplot table [x=nE,y=effInd] {data/ExampleStokesSmooth_1e-01.dat};
                \addplot table [x=nE,y=effInd] {data/ExampleStokesSmooth_1e-02.dat};
                \addplot table [x=nE,y=effInd] {data/ExampleStokesSmooth_1e-03.dat};
    \end{groupplot}
\end{tikzpicture}
\caption{Errors and estimator for $\alpha = 10^{-2}$ (left) and effectivity indices (right) for the problem from Section~\ref{sec:num_examples:stokes}. The black dotted line indicates $\mathcal{O}( (\#\mathscr{T})^{-1/2})$.}\label{fig:Stokes}
\end{figure}

Figure~\ref{fig:Stokes} shows the error in the field variables in the state, the error of the adjoint state, and the $L^2(\Omega)$-error of the control compared to the total error estimator $\eta$. Given that the exact solutions are smooth our a priori error analysis predicts $\mathcal{O}(h)$ convergence which is also seen in our experiment. 
Furthermore, the right plot of Figure~\ref{fig:Stokes} shows the effectivity index (defined as in Section~\ref{sec:num_examples:poisson}) for $\alpha =10^{-j}$, $j=0,1,2,3$.


\bibliographystyle{siamplain}
\bibliography{biblio}

\begin{thebibliography}{10}

\bibitem{MR3279496}
{\sc D.~Broersen and R.~Stevenson}, {\em A robust {P}etrov-{G}alerkin
  discretisation of convection-diffusion equations}, Comput. Math. Appl., 68
  (2014), pp.~1605--1618, \url{https://doi.org/10.1016/j.camwa.2014.06.019}.

\bibitem{MR3215064}
{\sc C.~Carstensen, L.~Demkowicz, and J.~Gopalakrishnan}, {\em A posteriori
  error control for {DPG} methods}, SIAM J. Numer. Anal., 52 (2014),
  pp.~1335--1353, \url{https://doi.org/10.1137/130924913}.

\bibitem{MR3521055}
{\sc C.~Carstensen, L.~Demkowicz, and J.~Gopalakrishnan}, {\em Breaking spaces
  and forms for the {DPG} method and applications including {M}axwell
  equations}, Comput. Math. Appl., 72 (2016), pp.~494--522,
  \url{https://doi.org/10.1016/j.camwa.2016.05.004}.

\bibitem{MR977489}
{\sc M.~Dauge}, {\em Stationary {S}tokes and {N}avier-{S}tokes systems on two-
  or three-dimensional domains with corners. {I}. {L}inearized equations}, SIAM
  J. Math. Anal., 20 (1989), pp.~74--97, \url{https://doi.org/10.1137/0520006}.

\bibitem{MR3308473}
{\sc J.~C. De~los Reyes}, {\em Numerical {PDE}-constrained optimization},
  SpringerBriefs in Optimization, Springer, Cham, 2015,
  \url{https://doi.org/10.1007/978-3-319-13395-9}.

\bibitem{MR2440722}
{\sc J.~C. de~Los~Reyes, P.~Merino, J.~Rehberg, and F.~Tr\"{o}ltzsch}, {\em
  Optimality conditions for state-constrained {PDE} control problems with
  time-dependent controls}, Control Cybernet., 37 (2008), pp.~5--38.

\bibitem{MR2630162}
{\sc L.~Demkowicz and J.~Gopalakrishnan}, {\em A class of discontinuous
  {P}etrov-{G}alerkin methods. {P}art {I}: the transport equation}, Comput.
  Methods Appl. Mech. Engrg., 199 (2010), pp.~1558--1572,
  \url{https://doi.org/10.1016/j.cma.2010.01.003}.

\bibitem{MR2837484}
{\sc L.~Demkowicz and J.~Gopalakrishnan}, {\em Analysis of the {DPG} method for
  the {P}oisson equation}, SIAM J. Numer. Anal., 49 (2011), pp.~1788--1809,
  \url{https://doi.org/10.1137/100809799}.

\bibitem{MR2743600}
{\sc L.~Demkowicz and J.~Gopalakrishnan}, {\em A class of discontinuous
  {P}etrov-{G}alerkin methods. {II}. {O}ptimal test functions}, Numer. Methods
  Partial Differential Equations, 27 (2011), pp.~70--105,
  \url{https://doi.org/10.1002/num.20640}.

\bibitem{MR4090394}
{\sc L.~Demkowicz, J.~Gopalakrishnan, and B.~Keith}, {\em The {DPG}-star
  method}, Comput. Math. Appl., 79 (2020), pp.~3092--3116,
  \url{https://doi.org/10.1016/j.camwa.2020.01.012}.

\bibitem{MR3095479}
{\sc L.~Demkowicz and N.~Heuer}, {\em Robust {DPG} method for
  convection-dominated diffusion problems}, SIAM J. Numer. Anal., 51 (2013),
  pp.~2514--2537, \url{https://doi.org/10.1137/120862065}.

\bibitem{MR3977484}
{\sc T.~F\"{u}hrer}, {\em Superconvergent {DPG} methods for second-order
  elliptic problems}, Comput. Methods Appl. Math., 19 (2019), pp.~483--502,
  \url{https://doi.org/10.1515/cmam-2018-0250}.

\bibitem{2023arXiv230113021F}
{\sc T.~{F{\"u}hrer} and N.~{Heuer}}, {\em {Robust DPG Fortin operators}},
  arXiv e-prints,  (2023), \url{https://doi.org/10.48550/arXiv.2301.13021}.

\bibitem{LSQoptimal}
{\sc T.~F\"{u}hrer and M.~Karkulik}, {\em Least-squares finite elements for
  distributed optimal control problems}, Numer. Math., accepted for
  publication, preprint available at arXiv:2210.16377 (2023),
  \url{https://arxiv.org/abs/2210.16377}.

\bibitem{OptControlHeatGS22}
{\sc G.~Gantner and R.~Stevenson}, {\em Applications of a space-time {FOSLS}
  formulation for parabolic {PDE}s}, IMA J. Numer. Anal., published online
  (2023), \url{https://doi.org/10.1093/imanum/drad012}.

\bibitem{MR3143683}
{\sc J.~Gopalakrishnan and W.~Qiu}, {\em An analysis of the practical {DPG}
  method}, Math. Comp., 83 (2014), pp.~537--552,
  \url{https://doi.org/10.1090/S0025-5718-2013-02721-4}.

\bibitem{MR775683}
{\sc P.~Grisvard}, {\em Elliptic problems in nonsmooth domains}, vol.~24 of
  Monographs and Studies in Mathematics, Pitman (Advanced Publishing Program),
  Boston, MA, 1985.

\bibitem{MR2122182}
{\sc M.~Hinze}, {\em A variational discretization concept in control
  constrained optimization: the linear-quadratic case}, Comput. Optim. Appl.,
  30 (2005), pp.~45--61, \url{https://doi.org/10.1007/s10589-005-4559-5}.

\bibitem{MR2516528}
{\sc M.~Hinze, R.~Pinnau, M.~Ulbrich, and S.~Ulbrich}, {\em Optimization with
  {PDE} constraints}, vol.~23 of Mathematical Modelling: Theory and
  Applications, Springer, New York, 2009.

\bibitem{MR2495058}
{\sc M.~Hinze, N.~Yan, and Z.~Zhou}, {\em Variational discretization for
  optimal control governed by convection dominated diffusion equations}, J.
  Comput. Math., 27 (2009), pp.~237--253.

\bibitem{KKT03}
{\sc T.~K\"{a}rkk\"{a}inen, K.~Kunisch, and P.~Tarvainen}, {\em Augmented
  {L}agrangian active set methods for obstacle problems}, J. Optim. Theory
  Appl., 119 (2003), pp.~499--533,
  \url{https://doi.org/10.1023/B:JOTA.0000006687.57272.b6}.

\bibitem{MR4334615}
{\sc B.~Keith}, {\em A priori error analysis of high-order {LL}* ({FOSLL}*)
  finite element methods}, Comput. Math. Appl., 103 (2021), pp.~12--18,
  \url{https://doi.org/10.1016/j.camwa.2021.10.015}.

\bibitem{MR3980263}
{\sc B.~Keith, A.~Vaziri~Astaneh, and L.~F. Demkowicz}, {\em Goal-oriented
  adaptive mesh refinement for discontinuous {P}etrov-{G}alerkin methods}, SIAM
  J. Numer. Anal., 57 (2019), pp.~1649--1676,
  \url{https://doi.org/10.1137/18M1181754}.

\bibitem{MR0404849}
{\sc R.~B. Kellogg and J.~E. Osborn}, {\em A regularity result for the {S}tokes
  problem in a convex polygon}, J. Functional Analysis, 21 (1976),
  pp.~397--431.

\bibitem{MR1786735}
{\sc D.~Kinderlehrer and G.~Stampacchia}, {\em An introduction to variational
  inequalities and their applications}, vol.~31 of Classics in Applied
  Mathematics, Society for Industrial and Applied Mathematics (SIAM),
  Philadelphia, PA, 2000, \url{https://doi.org/10.1137/1.9780898719451}.

\bibitem{MR3212590}
{\sc K.~Kohls, A.~R\"{o}sch, and K.~G. Siebert}, {\em A posteriori error
  analysis of optimal control problems with control constraints}, SIAM J.
  Control Optim., 52 (2014), pp.~1832--1861,
  \url{https://doi.org/10.1137/130909251}.

\bibitem{MR1301452}
{\sc V.~A. Kozlov, V.~G. Maz'ya, and C.~Schwab}, {\em On singularities of
  solutions to the {D}irichlet problem of hydrodynamics near the vertex of a
  cone}, J. Reine Angew. Math., 456 (1994), pp.~65--97.

\bibitem{MR0271512}
{\sc J.-L. Lions}, {\em Optimal control of systems governed by partial
  differential equations}, Die Grundlehren der mathematischen Wissenschaften,
  Band 170, Springer-Verlag, New York-Berlin, 1971.
\newblock Translated from the French by S. K. Mitter.

\bibitem{MR2647025}
{\sc G.~Lube and B.~Tews}, {\em Optimal control of singularly perturbed
  advection-diffusion-reaction problems}, Math. Models Methods Appl. Sci., 20
  (2010), pp.~375--395, \url{https://doi.org/10.1142/S0218202510004271}.

\bibitem{MR2641539}
{\sc V.~Maz'ya and J.~Rossmann}, {\em Elliptic equations in polyhedral
  domains}, vol.~162 of Mathematical Surveys and Monographs, American
  Mathematical Society, Providence, RI, 2010,
  \url{https://doi.org/10.1090/surv/162}.

\bibitem{MR163054}
{\sc J.~Ne\v{c}as}, {\em Sur une m\'{e}thode pour r\'{e}soudre les
  \'{e}quations aux d\'{e}riv\'{e}es partielles du type elliptique, voisine de
  la variationnelle}, Ann. Scuola Norm. Sup. Pisa Cl. Sci. (3), 16 (1962),
  pp.~305--326.

\bibitem{MR0483555}
{\sc P.-A. Raviart and J.~M. Thomas}, {\em A mixed finite element method for
  2nd order elliptic problems}, in Mathematical aspects of finite element
  methods ({P}roc. {C}onf., {C}onsiglio {N}az. delle {R}icerche ({C}.{N}.{R}.),
  {R}ome, 1975), Lecture Notes in Math., Vol. 606, Springer, Berlin, 1977,
  pp.~292--315.

\bibitem{MR2583281}
{\sc F.~Tr\"{o}ltzsch}, {\em Optimal control of partial differential
  equations}, vol.~112 of Graduate Studies in Mathematics, American
  Mathematical Society, Providence, RI, 2010,
  \url{https://doi.org/10.1090/gsm/112}.
\newblock Theory, methods and applications, Translated from the 2005 German
  original by J\"{u}rgen Sprekels.

\end{thebibliography}

\end{document}